\documentclass[12pt]{amsart}
\usepackage[nobysame,abbrev,alphabetic]{amsrefs}
\usepackage{amssymb}
\usepackage[all]{xy}
\usepackage{stmaryrd}
\usepackage[T2A,T1]{fontenc}

\DeclareMathOperator{\Aut}{Aut}
\DeclareMathOperator{\Char}{Char}
\DeclareMathOperator{\Gal}{Gal}
\DeclareMathOperator{\id}{id}
\DeclareMathOperator{\Id}{Id}
\DeclareMathOperator{\Img}{Im}

\DeclareMathOperator{\Ker}{Ker}
\DeclareMathOperator{\res}{res}

\DeclareMathOperator{\trg}{trg}

\newcommand{\Cl}{\mathrm{Cl}}
\newcommand{\Div}{\mathrm{Div}}
\newcommand{\discup}{\ \ensuremath{\mathaccent\cdot\cup}}

\newcommand{\nek}{,\ldots,}
\newcommand{\inv}{^{-1}}
\newcommand{\isom}{\cong}
\newcommand{\ndiv}{\hbox{$\,\not|\,$}}
\newcommand{\ord}{\mathrm{ord}}
\newcommand{\pr}{\mathrm{pr}}

\newcommand{\val}{\mathrm{val}}

\newtheorem{thm}{Theorem}[section]
\newtheorem{cor}[thm]{Corollary}
\newtheorem{lem}[thm]{Lemma}
\newtheorem{prop}[thm]{Proposition}
\newtheorem{defin}[thm]{Definition}
\newtheorem{exam}[thm]{Example}
\newtheorem{rem}[thm]{Remark}

\swapnumbers

\newtheorem{examples}[thm]{Examples}
\numberwithin{equation}{section}

\newcommand{\alp}{\alpha}
\newcommand{\gam}{\gamma}
\newcommand{\Gam}{\Gamma}
\newcommand{\del}{\delta}
\newcommand{\Del}{\Delta}
\newcommand{\eps}{\epsilon}
\newcommand{\lam}{\lambda}
\newcommand{\Lam}{\Lambda}
\newcommand{\sig}{\sigma}

\newcommand{\dbF}{\mathbb{F}}

\newcommand{\dbI}{\mathbb{I}}

\newcommand{\dbQ}{\mathbb{Q}}
\newcommand{\dbR}{\mathbb{R}}
\newcommand{\dbU}{\mathbb{U}}
\newcommand{\dbV}{\mathbb{V}}
\newcommand{\dbZ}{\mathbb{Z}}

\newcommand{\grm}{\mathfrak{m}}

\newcommand{\grq}{\mathfrak{q}}

\newcommand{\calG}{\mathcal{G}}

\newcommand{\CyrB}{\mbox{\usefont{T2A}{\rmdefault}{m}{n}\CYRB}}

\begin{document}

\title[Linking Invariants for Fields]{Linking Invariants for Valuations and Orderings on Fields}

\author{Ido Efrat}
\address{Earl Katz Family Chair in Pure Mathematics\\
Department of Mathematics\\
Ben-Gurion University of the Negev\\
P.O.\ Box 653, Be'er-Sheva 84105\\
Israel} \email{efrat@bgu.ac.il}

\thanks{This work was supported by the Israel Science Foundation (grant No.\ 569/21). \\
Data sharing not applicable to this article as no data sets were generated or analysed during the current study. 
Author declares no conflict of interests.}



\begin{abstract}
The mod-$2$ arithmetic Milnor invariants, introduced by Morishita, provide a decomposition law for primes in canonical Galois extensions of $\dbQ$ with unitriangular Galois groups, and contain the Legendre and R\'edei symbols
as special cases.
Morishita further proposed a notion of mod-$q$ arithmetic Milnor invariants, where $q$ is a prime power, for number fields containing the $q$th roots of unity and satisfying certain class field theory assumptions.
We extend this theory from the number field context to general fields, by introducing a notion of a linking invariant for discrete valuations and orderings.
We further express it as a Magnus homomorphism coefficient, and relate it to Massey product elements in Galois cohomology.
\end{abstract}

\maketitle

\tableofcontents

\section{Introduction}
\subsection{Background}
The classical Legendre symbol $(p_1/p_2)$ for distinct prime numbers $p_1,p_2$ describes the decomposition behavior of $p_2$ in the quadratic number field $\dbQ(\sqrt{p_1})$.
On 1939 R\'edei \cite{Redei39} introduced a triple symbol $[p_1,p_2,p_3]\in\{\pm1\}$, for odd primes $p_1,p_2,p_3$, which describes the decomposition pattern of $p_3$ in a certain Galois extension $K_{\{p_1,p_2\}}$ of $\dbQ$ with a dihedral Galois group of order $8$.
In the early 2000s, Morishita introduced a far-reaching generalization of these symbols, the \textsl{mod-$2$ arithmetic Milnor invariant} (\cite{Morishita02}, \cite{Morishita04}):
Given primes numbers $p_1\nek p_n,p_{n+1}$, $n\geq1$, this symbol describes the decomposition law of $p_{n+1}$ in a certain $\dbU_n(\dbF_2)$-Galois extension of $\dbQ$, associated with $p_1\nek p_n$.
Here $\dbU_n(R)$ denotes the group of all unipotent upper-triangular $(n+1)\times (n+1)$-matrices over the unital ring $R$.
Thus the Legendre and R\'edei symbols are the first and second symbols in this infinite sequence of arithmetic objects.
A candidate for the next step in this sequence, namely a 4-fold symbol with values in $\{\pm1\}$, was constructed by Amano \cite{Amano14b}.

More generally, for a prime power $q$, Morishita proposed to define \textsl{mod-$q$ arithmetic Milnor invariants} over number fields $F$ with values in $\dbZ/q$.
Here $F$ is assumed to contain the $q$th roots of unity and to satisfy certain assumptions arising from class field theory; in particular, the class number $|\Cl(F)|$ should be prime to $q$.
An explicit example was constructed by Amano, Mizusawa and Morishita \cite{AmanoMizusawaMorishita18}, who define a 3-fold mod-$3$ symbol over the Eisenstein number field $\dbQ(\zeta_3)=\dbQ(\sqrt{-3})$.

Morishita's construction was motivated by the analogy between prime numbers in $\dbQ$ and knots in the $3$-sphere $S^3$. 
This general philosophy, known as \textsl{arithmetic topology}, was suggested by Mumford, Manin and Mazur, and later developed by Kapranov, Morishita, Reznikov, and others.
We refer to \cite{Morishita02} and \cite{Morishita04} for more information on the development of these ideas.
Indeed, in knot theory Milnor defined similar symbols, called \textsl{higher linking invariants} \cite{Milnor57}, for links in $S^3$, and more generally, 
over \textsl{$q$-homology $3$-spheres}, that is, closed oriented $3$-manifolds $M$ with $H_1(M,\dbZ/q)=0$.
Morishita's invariants are their arithmetic analog over $\dbQ$, and over number fields as above, respectively.

\subsection{Contents of the present work}
In this paper, we extend parts of Morishita's theory from the context of number fields to the context of \textsl{general fields} $F$ containing the roots of unity of order $q=p^s$, with $p$ prime.
In this generalization, finite and infinite primes are replaced by discrete valuations and orderings, respectively, where the latter are considered only when $q=2$.

Morishita's invariants are closely related to the maximal pro-$p$ Galois group of the number field $F$ that is unramified outside a given finite list of primes. 
In the general field context, we take a set $\widehat I$ of discrete valuations or orderings on $F$ (when $q=2$).
Factoring $G_F(p)$ by its closed normal subgroup generated by the inertia subgroups of the discrete valuations not in $\widehat I$, and the $\dbZ/2$-subgroups corresponding to the orderings not in $\widehat I$, we obtain a Galois group $G_{F,{\widehat I}}(p)$.
We further assume that this group is generated by the images of the inertia groups and $\dbZ/2$-subgroups corresponding to the valuations and orderings, respectively, in a subset $I$ of $\widehat I$ consisting of $n$ independent valuations and orderings (this is the analog of the condition $p\ndiv|\Cl(F)|$ in the number field case).
For a discrete valuation $v=v_j$ (resp., an ordering $P=P_j$), with $j\not\in {\widehat I}$, we construct a \textsl{linking invariant} $[{\widehat I},I,j]$ (see Section \ref{section on Kummer linking structures} for the indexing convention).
It is a pro-$p$ group homomorphism
\[
[{\widehat I},I,j]\colon Z_j\to\dbZ/q,
\]
where $Z_j$ is a decomposition group (resp., a Galois group of a relative real closure) in $G_F(p)$ associated with $v$ (resp., $P$); see Section \ref{section on Kummer linking structures}.
In the main examples, where $F=\dbQ$, we take $I$ to be a finite set of prime numbers and take $\widehat I$ to be this set together with the unique ordering on $\dbQ$.
Then the linking invariant, when properly interpreted, contains the Legendre and R\'edei symbols as special cases (Section \ref{section on number fields}).
As in the works of Milnor and Morishita, $[{\widehat I},I,j]$ is defined only under the assumption that certain linking invariants for proper subsets of $\widehat I$ and $I$ vanish.

The homomorphism $[{\widehat I},I,j]$ is obtained from a canonical pro-$p$ group epimorphism 
\[
\rho\colon G_{F,{\widehat I}}(p)\to\dbU_n(\dbZ/q),
\]
by restricting it to $Z_j$ and projecting into the upper-right entry of $\dbU_n(\dbZ/q)$.
This homomorphism corresponds via Galois theory to a $\dbU_n(\dbZ/q)$-Galois extension $E_\rho$ of $F$.
Among the main properties of $E_\rho$ and the linking invariant $[{\widehat I},I,j]$ we have (see Subsection \ref{subsection on ramification and decomposition}):
\begin{itemize}
\item
The discrete valuations in $I$ ramify in $E_\rho$ with ramification groups $\dbZ/q$, whereas discrete valuations on $F$ outside $\widehat I$ are unramified in $E_\rho$.
\item
The orderings in $I$ do not extend to $E_\rho$, whereas orderings on $F$ outside $\widehat I$ extend to $E_\rho$.
\item
If $j$ corresponds to a valuation $v$, then $[{\widehat I},I,j]=0$ if and only if $v$ is completely decomposed in $E_\rho$.
\item
If $j$ corresponds to an ordering, then $[{\widehat I},I,j]=0$. 
\end{itemize}

We further define a notion of a \textsl{linking homomorphism} associated to two discrete valuations (Definition \ref{linking homomorphism}).
It is the field-theoretic generalization of Morishita's \textsl{linking numbers} of two finite primes in a number field and is, in turn, analogous to the linking number of two knots in the topological context. 

As conjectured by Stalling and proved by Turaev \cite{Turaev79} and Porter \cite{Porter80}, the Milnor invariants in knot theory can be interpreted cohomologically as Massey products, which are cohomological operations extending the cup product.
It was shown by Morishita \cite{Morishita04} that his arithmetic Milnor invariants also have such an interpretation.
This may be viewed as a generalization of the well-known interpretation of Legendre symbols as cup products.
We similarly relate our generalized linking invariants to Massey product elements in Galois cohomology -- see Section \ref{section on cohomological interpretation of linking invariants}.

\subsection{Strategy}
Needless to say, our approach is generally inspired by the above-mentioned works of Morishita.
Further, it is based on an abstract algebraic notion, which we call a \textsl{pro-$p$ linking structure}.
Loosely speaking, it consists of a pro-$p$ group $G$, families of closed subgroups $Z_i\trianglerighteq T_i$ of $G$, as well as information on the required image of the $T_i$'s in a certain target pro-$p$ group $U$.
A \textsl{globalization} for the linking structure is a homomorphism $G\to U$ which is consistent with this local information. 
In our field-theoretic applications, we apply this to $G=G_{F,{\widehat I}}(p)$ and the decomposition and inertia subgroups $Z_i,T_i$, $i\in I$, respectively (in the valuation case), or to the Galois groups of relative real closures (in the case of orderings in $I$). 
This gives rise to a linking structure $\calG_{{\widehat I},I}$ with target group $U=\dbU_n(\dbZ/q)$, which we call the ($I$-indexed) \textsl{Kummer linking structure of $F$ unramified outside $\widehat I$}.
A globalization for this structure gives rise to the homomorphism $\rho$, whence to the Galois extension $E_\rho$ and to the linking invariants $[{\widehat I},I,j]$.

The actual construction of our linking invariants is by induction on $n=|I|$:  
For $n=1$ the globalization is essentially a Kummer homomorphism of a uniformizer for the discrete valuation in $I$, or of a negative element in the case of orderings.
The induction step uses a fiber product construction for the groups $\dbU_n(\dbZ/q)$.
Moreover, it is based on the assumption that the underlying group $G_{{\widehat I}}$ of $\calG_{{\widehat I},I}$ has a certain presentation using generators and relations, which we call a group of \textsl{link type}.
This group structure extends results by Koch and Hoechsmann on the structure of $G_{F,{\widehat I}}(p)$ in the number field case, and follows Morishita's approach  (see Theorem \ref{KochHoechsmann}).
It is analogous to link groups in the knot theory.

\subsection{The Structure of the paper}
Section \ref{section on unitriangular matrices} provides some basic facts about the unitriangular group $\dbU_n(R)$, and quotients of it consisting of partial matrices.
Section \ref{section on linking structures} introduces linking structures and their globalizations in an abstract pro-$p$ context, as well as the notion of a linking structure $\calG$ of $V$-link type, where $V$ is a closed normal subgroup of the target group of $\calG$.
We show that in this case, and under a certain local triviality condition, any globalization for $\calG/V$ lifts to a globalization for $\calG$ (Proposition \ref{lifting construction}).
Section \ref{section on the structures GI} introduces the linking structure $\calG_{{\widehat I},I}$, which is an abstract pro-$p$ analog of the maximal pro-$p$ Galois group $G_{F,{\widehat I}}(p)$ of a number field unramified outside a finite set $\widehat I$ of primes.
We show how globalizations for such linking structures can be combined using a fiber product construction for unitriangular groups, to obtain globalizations for structures $\calG_{{\widehat I},I}$ with larger sets ${\widehat I},I$.
Globalizations for $\calG_{{\widehat I},I}$ give rise to our main notion of the \textsl{linking invariant} $[{\widehat I},I,j]$.
Our approach here is different from that of \cite{Morishita02} and \cite{Morishita04}, where the arithmetic Milnor invariants are defined as coefficients of the Magnus homomorphism of certain Frobenius lifts.
However, in Section \ref{section on linking invariants and Magnus theory} we show that, assuming that the groups $T_i$ are free pro-$p$-cyclic, $[{\widehat I},I,j]$ can also be interpreted in the relevant cases in terms of Magnus homomorphisms. 

Then we turn to the arithmetic situation and recall in Section \ref{section on valuations and orderings} the needed notions and facts on valuations and orderings on fields.
Crucial to our approach is the connection between inertia groups of valuations and the Kummer pairing, which is discussed in Section \ref{section on inertia group and the Kummer pairing}.
Then we introduce in Section \ref{section on Kummer linking structures} our main construction of \textsl{Kummer linking structures} of fields with respect to a sets $I\subseteq {\widehat I}$ of valuations and orderings.
We explain the meaning of the linking invariant $[{\widehat I},I,j]$ in this arithmetic context in terms of ramification and decomposition of valuations and orderings in the Galois extension $E_\rho$ associated with a globalization $\rho$ for $\calG_{{\widehat I},I}$.
We also define the \textsl{linking homomorphism} ${\rm lk}(v_j,v_i)$ of two discrete valuations.

In Sections \ref{section on number fields} and \ref{section on classical invariants over Q} we focus on number fields.
We interpret some of the main results of Morishita and Amano in the language of linking structures and globalizations, and in particular, interpret the Legendre and R\'edei symbols as linking invariants arising from globalizations of Kummer structures.
Here $I$ consists of discrete valuations corresponding to primes numbers $p_i\equiv1\pmod4$, and $\widehat I$ consists of $I$ together with the unique ordering on $\dbQ$. 

Finally, in Section \ref{section on cohomological interpretation of linking invariants} we go back to the general pro-$p$ setting of the linking structures $\calG_{{\widehat I},I}$, and interpret the linking invariant $[{\widehat I},I,j]$  cohomologically in terms of Massey product elements. 

I thank Masanori Morishita and the referees for their important comments on earlier versions of this paper, and in particular regarding the formulation of Theorem \ref{KochHoechsmann}.

\section{Unitriangular Matrices}
\label{section on unitriangular matrices}
\subsection{Preliminaries}
\label{subsection on preliminaries}
We first recall some notions and facts from \cite{Efrat22}.
Fix a non-negative integer $n$ and let 
\[
\dbI_n=\{(k,l)\ |\ 1\leq k\leq l\leq n+1\}, \quad
\Del_n=\{(k,k)\ |\ 1\leq k\leq n+1\}.
\]
We view $\dbI_n$ and $\Del_n$ as the sets of upper-right entries, resp., the diagonal, of an $(n+1)\times(n+1)$-matrix.

A subset $S$ of $\dbI_n$ is called \textsl{convex} if:
\begin{enumerate}
\item[(i)]
$\Del_n\subseteq S$;
\item[(ii)]
For every $(k,l)\in S$ and $(k',l')\in\dbI_n$ such that $k\leq k'$ and $l'\leq l$ also $(k',l')\in S$. 
\end{enumerate}
Loosely speaking, (ii) means that $S$ has no ``holes'' towards the diagonal.

We fix a profinite unital commutative ring $R$ with additive group $R^+$.
For a convex subset $S$ of $\dbI_n$, we write $\dbU_S$ for the set of all partial matrices $(a_{kl})_{(k,l)\in S}$ with $a_{kl}\in R$ and $a_{kk}=1$ for $1\leq k\leq n$.
It forms a group under the multiplication map
\[
(a_{kl})_{(k,l)\in S}\cdot (b_{kl})_{(k,l)\in S}=\Bigl(\sum_{r=k}^la_{kr}b_{rl}\Bigr)_{(k,l)\in S}.
\]
We write $\Id_{\dbU_S}=(\del_{kl})_{(k,l)\in S}$ for the identity element in $\dbU_S$.
In particular, 
\[
\dbU_n=\dbU_{\dbI_n}
\]
may be identified with the group of all unitriangular (i.e., unipotent and upper-triangular) $(n+1)\times(n+1)$-matrices over $R$.

For $(k,l)\in S$ let $\pr_{kl}\colon \dbU_S\to R$ be the projection map on the $(k,l)$-entry.

Given $(k,l)\in S\setminus\Del_n$ and $r\in R$, we write $\Id_S+rE_{kl}$ for the matrix in $\dbU_S$ which is $1$ on $\Del_n$ and $r$ at entry $(k,l)$, and is $0$ at all other entries.
We write $\Id_S+RE_{kl}$ for the set of all such matrices with $r\in R$. 

\begin{lem}
\label{elementary matrices}
\begin{enumerate}
\item[(a)]
For every $(k,l)\in S\setminus\Del_n$ the subset $\Id_S+RE_{kl}$ forms a subgroup of $\dbU_S$ which is isomorphic to $R^+$.
\item[(b)]
The partial matrices $\Id_{\dbU_n}+E_{k,k+1}$, $k=1,2\nek n$, generate $\dbU_n$.
\end{enumerate}
\end{lem}
\begin{proof}
(a) is straightforward.
For (b) see \cite{Weir55}*{\S1}.
\end{proof}

\subsection{Projection homomorphisms}
\begin{lem}
\label{ttt}
Let $1\leq m\leq n$ and let $\iota\colon \{1,2\nek m+1\}\to\{1,2\nek n+1\}$ be a strictly increasing map.
Let $S'\subseteq S$ be convex subsets of $\dbI_n$ such that every $1\leq r\leq n$ with $(r,r+1)\in S'$ is in the image of $\iota$.
Then:
\begin{enumerate}
\item[(a)]
The subset $\bar S=\{(k,l)\in\dbI_m\ |\ (\iota(k),\iota(l))\in S'\}$ is convex in $\dbI_m$;
\item[(b)]
The map 
\[
u\colon \dbU_S\to\dbU_{\bar S}, \quad (a_{ij})_{(i,j)\in S}\mapsto\bigl(a_{\iota(k),\iota(l)}\bigr)_{(k,l)\in \bar S},
\]
is a group homomorphism.
\end{enumerate}
\end{lem}
\begin{proof}
(a) \quad
Let $(k,l)\in \bar S$ and $(k',l')\in \dbI_m$ satisfy $k\leq k'$ and $l\geq l'$.
Then $\iota(k)\leq \iota(k')$ and $\iota(l)\geq \iota(l')$.
As $(\iota(k),\iota(l))\in S'$, the convexity of $S'$ implies that $(\iota(k'),\iota(l'))\in S'$, so $(k',l')\in \bar S$.   

\medskip

(b) \quad
First we claim that if $(k,l)\in \bar S$, then every integer $\iota(k)\leq r\leq\iota(l)$ is of the form $\iota(t)$ for some $k\leq t\leq l$;

This is immediate if $r=\iota(l)$.

If $\iota(k)\leq r<\iota(l)$, then the inclusion $(\iota(k),\iota(l))\in S'$ and the convexity of $S'$ imply that $(r,r+1)\in S'$.
By assumption, this implies that $r=\iota(t)$ for some $1\leq t\leq m+1$.
Since $\iota$ is strictly increasing, $k\leq t<l$, as required.

Now for $(a_{ij})_{(i,j)\in S},(b_{ij})_{(i,j)\in S}\in \dbU_S$ we compute using the above claim:
\[
\begin{split}
u\Bigl(\bigl(a_{ij}\bigr)_{(i,j)\in S}\cdot \bigl(b_{ij}\bigr)_{(i,j)\in S}\Bigr)
&=u\Bigl(\bigl(\sum_{i\leq r\leq j}a_{ir}b_{rj}\bigr)_{(i,j)\in S}\Bigr)\\
&=\Bigl(\sum_{\iota(k)\leq r\leq \iota(l)}a_{\iota(k),r}b_{r,\iota(l)}\Bigr)_{(k,l)\in \bar S}\\
&=\Bigl(\sum_{k\leq t\leq l}a_{\iota(k),\iota(t)}b_{\iota(t),\iota(l)}\Bigr)_{(k,l)\in \bar S}\\
&=(a_{\iota(k),\iota(l)})_{(k,l)\in \bar S}\cdot (b_{\iota(k),\iota(l)})_{(k,l)\in \bar S}\\
&=u\bigl((a_{ij})_{(i,j)\in S}\bigr)\cdot u\bigl((b_{ij})_{(i,j)\in S}\bigr),
\end{split}
\]
as desired.
\end{proof}

\begin{examples}
\label{special projections are homomorphisms}
\rm
(1) \quad
Take $m= n$ and $\iota=\id$.
Then for every convex subsets $S'\subseteq S$ of $\dbI_n$, the projection map $\pr_{S,S'}\colon\dbU_S\to\dbU_{S'}$ is a group homomorphism.

\medskip

(2) \quad
Let $1\leq m\leq n$, $S=\dbI_n$ and $S'=\Del_n\cup\dbI_m$. 
Define the map $\iota\colon\{1,2\nek m+1\}\to \{1,2\nek n+1\}$ by $\iota(k)=k$.
Then $\bar S=\dbI_m$.
We obtain that the projection map on the upper-left $(m+1)\times (m+1)$-submatrix is a group homomorphism $p'_{n,m}\colon\dbU_n\to\dbU_m$ .

Similarly, the projection map on the lower-right $(m+1)\times (m+1)$-submatrix is a group homomorphism $p''_{n,m}\colon\dbU_n\to\dbU_m$ .

\medskip

(3) \quad
By combining the two maps in (2) (or directly), we see that for every $1\leq l\leq n$ the projection map $\pr_{l,l+1}\colon \dbU_n\to\dbU_1\isom R^+$ on the $(l,l+1)$-entry is a group homomorphism.
\end{examples}

\begin{lem}
\label{fiber product lemma}
Let $1\leq m_1,m_2\leq n$ satisfy $m_1+m_2\geq n$, and set $t=m_1+m_2-n$.
Also let $V=\Ker(p'_{n,m_1})\cap\Ker(p''_{n,m_2})$.
Then the maps $p'_{n,m_1}$ and $p''_{n,m_2}$ induce a group isomorphism
\[
\dbU_n/V=\dbU_{m_1}\times_{\dbU_t}\dbU_{m_2},
\]
where the fiber product is taken with respect to $p''_{m_1,t}$ and $p'_{m_2,t}$.
\end{lem}
\begin{proof}
We note that
$p''_{m_1,t}\circ p'_{n,m_1}=p'_{m_2,t}\circ p''_{n,m_2}$
is the projection to rows and columns $n-m_2+1\nek m_1$.
This gives rise to a homomorphism
\[
(p'_{n,m_1},p''_{n,m_2})\colon \dbU_n\to \dbU_{m_1}\times_{\dbU_t}\dbU_{m_2}
\]
with kernel $V$.

This homomorphism is surjective, since any pair $(M_1,M_2)\in\dbU_{m_1}\times_{\dbU_t}\dbU_{m_2}$ has a preimage $M$ in $\dbU_n$, which is $M_1$ on the upper-left $(m_1+1)\times(m_1+1)$-submatrix, is $M_2$ on the lower-right $(m_2+1)\times(m_2+1)$-submatrix, and is $0$ elsewhere.
\end{proof}

\subsection{The subgroups $V_S$}
For a convex subset $S$ of $\dbI_n$ set 
\[
V_S=\Ker\bigl(\pr_{\dbI_n,S}\colon \dbU_n\to\dbU_S\bigr).
\]
For $r\geq1$ consider the following convex subset of $\dbI_n$:
\[
\dbI(n,r)=\{(i,j)\in\dbI_n\ |\ j-i<r\}.
\]

\begin{exam}
\label{Sr}
\rm
\begin{enumerate}
\item[(1)]
One has $\dbI(n,1)=\Del_n$, so $\dbU_{\dbI(n,1)}=\{\Id\}$ and $V_{\dbI(n,1)}=\dbU_n$.
\item[(2)]
One has $\dbI(n,n)=\dbI_n\setminus\{(1,n+1)\}$.
We denote
\[
\overline{\dbU}_n:=\dbU_{\dbI(n,n)}, \quad
\dbV_n:=V_{\dbI(n,n)}=\Id_{\dbU_n}+RE_{1,n+1}\isom R^+,
\]
where the isomorphism is by Lemma \ref{elementary matrices}(a).
\item[(3)]
$\dbI(n,n+1)=\dbI_n$, so $\dbU_{\dbI(n,n+1)}=\dbU_n$ and $V_{\dbI(n,n+1)}=\{\Id_{\dbU_n}\}$.
\end{enumerate}
\end{exam}
For every positive integers $r_1,r_2$, the commutator map in $\dbU_n$ satisfies
\[
[V_{\dbI(n,r_1)},V_{\dbI(n,r_2)}]\subseteq V_{\dbI(n,r_1+r_2)},
\]
e.g., by \cite{Efrat22}*{Prop.\ 2.2}.
In particular,
\begin{equation}
\label{center of Un}
[\dbV_n,\dbU_n]=[V_{\dbI(n,n)},V_{\dbI(n,1)}]=V_{\dbI(n,n+1)}=\{\Id_{\dbU_n}\},
\end{equation}
that is, $\dbV_n$ lies in the center of $\dbU_n$.

\section{Linking Structures}
\label{section on linking structures}
\subsection{Globalizations of linking structures}
We fix a prime number $p$.
A (pro-$p$) \textsl{linking structure}
\begin{equation}
\label{linking structure}
\calG=(G,U,Z_i,T_i,\lam_i,\pi_i)_{i\in I}.
\end{equation} 
consists of the following data:
\begin{itemize}
\item
Pro-$p$ groups $G$ and $U$;
\item
An index set $I$;
\item
For each $i\in I$ a pro-$p$ group $Z_i$ and a closed normal subgroup $T_i$ of $Z_i$ with a fixed (homomorphic) section of the natural  epimorphism $Z_i\to \bar Z_i:=Z_i/T_i$ (i.e., a semi-direct product decomposition $Z_i=T_i\rtimes\bar Z_i$);
\item
For each $i\in I$, pro-$p$ group homomorphisms $\lam_i\colon Z_i\to G$ and $\pi_i\colon T_i\to U$ 
such that $G$ is generated as a pro-$p$ group by the images $\lam_i(T_i)$, $i\in I$.
\end{itemize}

A \textsl{globalization} for $\calG$ is a pro-$p$ group homomorphism $\rho\colon G\to U$ such that for every $i\in I$ there is a commutative square
\[
\xymatrix{
Z_i\ar[d]_{\lam_i}&T_i\ar@{^{(}->}[l]\ar[d]^{\pi_i}&\\
G\ar[r]^{\rho}&U.\\
}
\]

\begin{rem}
\label{uniqueness}
\rm
Since the subgroups $\lam_i(T_i)$, $i\in I$,  generate $G$, there is at most one globalization for $\calG$.
\end{rem}

\begin{rem}
\label{conjugates}
\rm
Suppose that for every $i\in I$ we have $Z_i\leq G$ and $\lam_i$ is the inclusion map.
Let $\sig_i\in G$, $i\in I$, and define homomorphisms 
\[
\begin{split}
&\lam_i^{\sig_i}\colon\sig_iZ_i\sig_i\inv\to G, \quad \sig_i\sig\sig_i\inv\mapsto\sig,\\
&\pi_i^{\sig_i}\colon\sig_iT_i\sig_i\inv\to U, \quad \sig_i\tau\sig_i\inv\mapsto\pi_i(\tau).
\end{split}
\]
We obtain a \textsl{conjugate linking system} 
\[
\calG^{(\sig_i)}=(G,U,\sig_iZ_i\sig_i\inv,\sig_iT_i\sig_i\inv,\lam_i^{\sig_i},\pi_i^{\sig_i})_{i\in I}.
\]
Any globalization $\rho$ for $\calG$ is a globalization for $\calG^{(\sig_i)}$.
\end{rem}

\subsection{Lifting of globalizations in link type structures}
\label{subsection on lifting of globalizations in link type structures}
Consider the free pro-$p$ product $\hat T=\star_{i\in I}T_i$ \cite{NeukirchSchmidtWingberg}*{Ch.\ IV, \S1}.
Let
\[
\hat\lam_T=\star_{i\in I}(\lam_i|_{T_i})\colon \hat T\to G, \quad \hat\pi=\star_{i\in I}\pi_i\colon \hat T\to U.
\]
Then $\hat\lam_T$ is surjective.
A pro-$p$ group homomorphism $\rho\colon G\to U$ is a globalization for $\calG$ if and only if $\rho\circ\hat\lam_T=\hat\pi$.
For $i\in I$ we further set 
\[
\hat Z_i:=\hat\lam_T\inv(\lam_i(\bar Z_i))\subseteq\hat T.
\]

\begin{defin}
\rm
Given a closed subgroup $V$ of $U$, we say that the linking structure $\calG$ has a $V$-\textsl{link type} if $\Ker(\hat\lam_T)$ is generated as a closed normal subgroup of $\hat T$ by elements of 
\[
\Ker(\hat\pi)\bigl[\hat\pi\inv[C_U(V)],\hat Z_i\bigr]
\]
for certain indices $i\in I$, where $C_U(V)$ denotes the centralizer of $V$ in $U$, and $[\cdot,\cdot]$ denotes the commutator closed subgroup.
\end{defin}

\begin{exam}
\label{example of link type group}
\rm
In our arithmetical applications, we will deal with the following setup:
Let $n$ be a positive integer, let $q>1$ be a $p$-power, and let $I=\{i_1\nek i_n\}$ have size $n$.
Take
\[
R=\dbZ/q, \quad U=\dbU_n, \quad V=\dbV_n.
\]
By (\ref{center of Un}), $C_U(V)=\dbU_n$.
Suppose further that $\Img(\pi_i)\subseteq \Id_{\dbU_n}+(\dbZ/q)E_{l,l+1}$ for every $i=i_l\in I$, so by Lemma \ref{elementary matrices}(a), $\hat\pi(\tau^q)=\Id_{\dbU_n}$ for every $\tau\in\hat T$. 
If $\Ker(\hat\lam_T)$ is generated as a closed normal subgroup of $\hat T$ by elements of the form 
\[
\tau^{q\alp}[\tau,\sig]
\]
for some $\alp\in\dbZ_p$, $\tau\in T_i$, and $\sig\in\hat Z_i$, $i\in I$, then  $\calG$ has a $V$-link type.
\end{exam}

Let $\calG$ be as in (\ref{linking structure}) and let $V$ be a closed normal subgroup of $U$ with projection map
$\pr_{U,U/V}\colon U\to U/V$.
There is a \textsl{quotient linking structure}
\[
\calG/V:=(G,U/V,Z_i,T_i,\lam_i,\pr_{U,U/V}\circ\pi_i)_{i\in I}.
\]

\begin{prop}
\label{lifting construction}
\begin{enumerate}
\item[(a)]
Every globalization $\rho$ for $\calG$ which maps $\lam_i(\bar Z_i)$, $i\in I$, into $V$ induces a globalization $\bar\rho$ for $\calG/V$ which is trivial on $\lam_i(\bar Z_i)$, $i\in I$.
\item[(b)]
Assume that $\calG$ has a $V$-link type.
Then every globalization $\bar\rho$ for $\calG/V$ which is trivial on $\lam_i(\bar Z_i)$, $i\in I$, lifts uniquely to a globalization $\rho$ for $\calG$, which moreover,  maps $\lam_i(\bar Z_i)$, $i\in I$, into $V$.
\end{enumerate}
\end{prop}
\begin{proof}
(a)$\Rightarrow$(b): \quad
This is immediate, taking $\bar\rho=\pr_{U,U/V}\circ\rho$.

\medskip

(b)$\Rightarrow$(a): \quad
For $i\in I$ we have 
\[
(\pr_{U/V}\circ\hat\pi)(\hat Z_i)=(\bar\rho\circ\hat\lam_T)(\hat Z_i)=(\bar\rho\circ\lam_i)(\bar Z_i)=1.
\]
Hence $\hat\pi(\hat Z_i)\subseteq V$.
It follows that
\[
\hat\pi\Bigl(\Ker(\hat\pi)\bigl[\hat\pi\inv[C_U(V)],\hat Z_i\bigr]\Bigr)=[C_U(V),V]=1.
\]
As $\calG$ has a $V$-link type, this implies that $\hat\pi(\Ker(\hat\lam_T))=1$.
Since $\hat\lam_T$ is an epimorphism, $\hat\pi$ therefore factors through a pro-$p$ group homomorphism $\rho\colon G\to U$, which is thus a globalization for $\calG$.

The uniqueness of $\rho$ is by Remark \ref{uniqueness}.

The assertion on the images of $\lam_i(\bar Z_i)$, $i\in I$, is immediate.
\end{proof}

\section{The Linking Structure $\calG_{\hat I,I}$}
\label{section on the structures GI}
In this section, we construct linking structures imitating the maximal pro-$p$ Galois group of a number field unramified outside a given set of primes. 
This will be applied in Section \ref{section on Kummer linking structures} in an arithmetic setting.
 \subsection{The construction of $\calG_{{\widehat I},I}$}
Let $R$ be a unital pro-$p$ ring, let $G$ be a pro-$p$ group, and let $Z_j,T_j$, $j\in J$, be closed subgroups of $G$ such that $T_j$ is a normal subgroup of $Z_j$ for every $j$.

Consider a nonempty finite subset $I=\{i_1\nek i_n\}$ of $J$ of $n$ elements such that $Z_i=T_i\rtimes\bar Z_i$, where $\bar Z_i=Z_i/T_i$, for every $i\in I$.
For an intermediate set $I\subseteq {\widehat I}\subseteq J$, let $N_{{\widehat I}}$ be the closed normal subgroup of $G$ generated by the subgroups $T_j$, $j\in J\setminus {\widehat I}$, and set $G_{{\widehat I}}=G/N_{{\widehat I}}$.

We make the following assumption:

\smallskip

(A) \quad 
The images of $\lam_i(T_i)$, $i\in I$, generate $G_{{\widehat I}}$.

\smallskip

For each $j\in J$ we fix a pro-$p$ group epimorphism 
\[
\bar\pi_j\colon T_j\to R^+,
\]
and define a pro-$p$ group homomorphism $\pi_j\colon T_j\to\dbU_n$ by 
\begin{equation}
\label{images of pi j}
\pi_j(\tau)_{lk}=\begin{cases} 1,& \hbox{if } l=k, \\
\bar\pi_j(\tau),&\hbox{if }j=i_l, \ k=l+1,\\
0,&\hbox{otherwise}.
\end{cases}
\end{equation}

Then we have a linking structure
\begin{equation}
\label{G hat I I}
\calG_{{\widehat I},I}=(G_{{\widehat I}},\dbU_n,Z_i,T_i,\res_i,\pi_i)_{i\in I},
\end{equation}
We further set $\calG_I=\calG_{I,I}$.
Note that if ${\widehat I}'\subseteq {\widehat I}$, then the images of $\lam_i(T_i)$, $i\in {\widehat I}'\cap I$, generate $G_{{\widehat I}'}$, so we also have a linking structure $\calG_{{\widehat I}',{\widehat I}'\cap I}$.

\begin{lem}
\label{surjectivity}
Every globalization $\rho$ for $\calG_{{\widehat I},I}$ is surjective.
\end{lem}
\begin{proof}
For every $i=i_l\in I$ we have $\rho(\res_i(T_i))=\pi_i(T_i)=\Id_{\dbU_n}+RE_{l,l+1}$, by (\ref{images of pi j}) and the surjectivity of $\bar\pi_i$.
Now apply Lemma \ref{elementary matrices}(b).
\end{proof}

First we examine globalizations for $\calG_{{\widehat I},I}$ in the special case $|{\widehat I}|=2$, $|I|=1$. 

\begin{prop}
\label{equivalent conditions for isolated}
Let $i_1,i_2\in J$, $i_1\neq i_2$.
The following conditions are equivalent:
\begin{enumerate}
\item[(a)]
There is a globalization $\rho$ for $\calG_{\{i_1,i_2\},\{i_1\}}$;
\item[(b)]
There is a pro-$p$ group epimorphism $\rho_0\colon G\to R^+$ which is $\bar\pi_{i_1}$ on $T_{i_1}$, and is trivial on every $T_j$ with $j\in J$, $j\neq i_1,i_2$.
\end{enumerate}
\end{prop}
\begin{proof}
(a) $\Rightarrow$(b): \quad
We take
\[
\rho_0=\rho\circ\res\colon G\to \dbU_1\isom R^+,
\]
where $\res$ is the restriction to $G_{\{i_1,i_2\}}$.

\medskip
(b)$\Rightarrow$(a): \quad
For $\rho_0\colon G\to R^+$ as in (b), consider the pro-$p$ homomorphism
\[
\hat\rho\colon G\to\dbU_1, \quad \sig\mapsto \Id+\rho_0(\sig)E_{12}.
\]
It is trivial on $T_j$ for every $j\in J$, $j\neq i_1,i_2$.
Hence it factors through a pro-$p$ group homomorphism $\rho\colon G_{\{i_1,i_2\}}\to\dbU_1$.
By (\ref{images of pi j}), $\rho\circ\res_{i_1}=\pi_{i_1}$ on $T_{i_1}$.
\end{proof}

\subsection{Globalizations in fiber products}
In our field-theoretic applications, we will construct globalizations for the linking structures $\calG_{{\widehat I},I}$ inductively using the fiber product construction of Lemma \ref{fiber product lemma}, as follows.

We write $I=\{i_1\nek i_n\}$, ${\widehat I}=\{i_1\nek i_n,i_{n+1}\nek i_N\}$, where $2\leq n\leq N=|{\widehat I}|$.
Let $1\leq m_1,m_2\leq n$ satisfy $m_1+m_2\geq n$ and denote $t=m_1+m_2-n$.
Consider the subsets 
\[
{\widehat I}_1=\{i_1\nek i_{m_1},i_{n+1}\nek i_N\}, \  {\widehat I}_2=\{i_{n-m_2+1}\nek i_n,i_{n+1}\nek i_N\},
\]
\[
{\widehat I}_{12}={\widehat I}_1\cap {\widehat I}_2=\{i_{n-m_2+1}\nek i_{m_1},i_{n+1}\nek i_N\}
\]
of $\widehat I$.
Thus $|{\widehat I}_1\cap I|=m_1$, $|{\widehat I}_2\cap I|=m_2$ and $|{\widehat I}_{12}\cap I|=t$.

Let 
\[
p'_{n,m_1}\colon\dbU_n\to\dbU_{m_1}, \ 
p''_{n,m_2}\colon\dbU_n\to\dbU_{m_2}, \
p'_{m_2,t}\colon\dbU_{m_2}\to\dbU_t, \   
p''_{m_1,t}\colon \dbU_{m_1}\to\dbU_t, 
\]
be the projection homomorphisms as in Section \ref{section on unitriangular matrices}.
We obtain linking structures 
\[
\begin{split}
\calG_{{\widehat I}_1,{\widehat I}_1\cap I}&=(G_{{\widehat I}_1},\dbU_{m_1},Z_i,T_i,\res_i,p'_{n,m_1}\circ\pi_i)_{i\in {\widehat I}_1\cap I},
\\
\calG_{{\widehat I}_2,{\widehat I}_2\cap I}&=(G_{{\widehat I}_2},\dbU_{m_2},Z_i,T_i,\res_i,p''_{n,m_2}\circ\pi_i)_{i\in {\widehat I}_2\cap I},
\\
\calG_{{\widehat I}_{12},{\widehat I}_{12}\cap I}&=(G_{{\widehat I}_{12}},\dbU_t,Z_i,T_i,\res_i,p''_{m_1,t}\circ p'_{n,m_1}\circ\pi_i)_{i\in {\widehat I}_{12}\cap I},
\end{split}
\]
where $\res_i$ denotes the restriction map from $Z_i$ to the relevant group.
Let 
\[
V=\Ker(p'_{n,m_1})\cap\Ker(p''_{n,m_2})\subseteq\dbU_n.
\]
 
\begin{thm}
\label{main lifting theorem}
Let $\rho_1\colon G_{{\widehat I}_1}\to\dbU_{m_1}$ and $\rho_2\colon G_{{\widehat I}_2}\to\dbU_{m_2}$ be globalizations for $\calG_{{\widehat I}_1,{\widehat I}_1\cap I}$, $\calG_{{\widehat I}_2,{\widehat I}_2\cap I}$, respectively.
Then:
\begin{enumerate}
\item[(a)]
There is a globalization 
\[
\bar\rho\colon G_{{\widehat I}}\to\dbU_{m_1}\times_{\dbU_t}\dbU_{m_2}
\]
for $\calG_{{\widehat I},I}/V$ which lifts both $\rho_1$ and $\rho_2$ with respect to the projection maps;
\item[(b)]
If in addition $\calG_{{\widehat I},I}$ has a $V$-link type and $\rho_1$, $\rho_2$ are trivial on $\res_i(\bar Z_i)$ for all $i\in I$, then there is a surjective globalization $\rho\colon G_{{\widehat I}}\to\dbU_n$ for $\calG_{{\widehat I},I}$ which lifts both $\rho_1$ and $\rho_2$ and such that $\rho(\res_i(\bar Z_i))\in V$ for all $i\in I$.
\end{enumerate}
\end{thm}
\begin{proof}
(a) \quad
By (\ref{images of pi j}), the homomorphism $p''_{m_1,t}\circ\rho_1\colon G_{{\widehat I}_1}\to\dbU_t$ is trivial on the images of $T_1\nek T_{n-m_2}$, 
and the homomorphism  $p'_{m_2,t}\circ\rho_2\colon G_{{\widehat I}_2}\to\dbU_t$ is trivial on the images of $T_{m_1+1}\nek T_n$.
Hence they induce globalizations $\rho'_1\colon G_{{\widehat I}_{12}}\to\dbU_t$, $\rho'_2\colon G_{{\widehat I}_{12}}\to\dbU_t$, respectively, for $\calG_{{\widehat I}_{12},{\widehat I}_{12}\cap I}$ (see diagram below).
By Remark \ref{uniqueness}, $\calG_{{\widehat I}_{12},{\widehat I}_{12}\cap I}$ has at most one globalization, so $\rho'_1=\rho'_2$.

Let $\hat\rho_1\colon G_{{\widehat I}}\to\dbU_{m_1}$ and $\hat\rho_2\colon G_{{\widehat I}}\to\dbU_{m_2}$ be the homomorphisms induced by $\rho_1$, $\rho_2$, respectively, so that the following diagrams commute:
\[
\xymatrix{
G_{{\widehat I}}\ar[d]_{\res}\ar[dr]^{\hat\rho_1}&&&G_{{\widehat I}}\ar[d]_{\res}\ar[dr]^{\hat\rho_2}&\\
G_{{\widehat I}_1}\ar[d]_{\res}\ar[r]^{\rho_1}&\dbU_{m_1}\ar[d]^{p''_{m_1,t}}&&G_{{\widehat I}_2}\ar[d]_{\res}\ar[r]^{\rho_2}&\dbU_{m_2}\ar[d]^{p'_{m_2,t}}\\
G_{{\widehat I}_{12}}\ar[r]^{\rho'_1}&\dbU_t,&&G_{{\widehat I}_{12}}\ar[r]^{\rho'_2}&\dbU_t.
}
\]
As $\rho'_1=\rho'_2$, we have $p''_{m_1,t}\circ\hat\rho_1=p'_{m_2,t}\circ\hat\rho_2$.
Using  Lemma \ref{fiber product lemma} we conclude that $\hat\rho_1$, $\hat\rho_2$ combine to a homomorphism 
\[
\bar\rho=(\hat\rho_1,\hat\rho_2)\colon G_{{\widehat I}}\to\dbU_{m_1}\times_{\dbU_t}\dbU_{m_2}=\dbU_n/V.
\]
There are commutative squares
\[
\xymatrix{
G_{{\widehat I}}\ar[d]_{\res}\ar[r]^{\bar\rho}&\dbU_n/V\ar[d]^{p'_{n,m_1}}&&G_{{\widehat I}}\ar[d]_{\res}\ar[r]^{\bar\rho}&\dbU_n/V\ar[d]^{p''_{n,m_2}}\\
G_{{\widehat I}_1}\ar[r]^{\rho_1}&\dbU_{m_1},&&G_{{\widehat I}_1}\ar[r]^{\rho_2}&\dbU_{m_2}.
}
\]
Hence $\bar\rho$ is a globalization for $\calG_{{\widehat I},I}/V$.

\medskip

(b) \quad
By the assumptions, $\bar\rho$ is trivial on each subgroup $\res_i(\bar Z_i)$, $i\in I$.
Proposition \ref{lifting construction}(b) yields a globalization $\rho\colon G_{{\widehat I}}\to\dbU_n$ for $\calG_{{\widehat I},I}$ which lifts $\bar\rho$ and such that $\rho(\res_i(\bar Z_i))\in V$ for every $i\in I$.
In particular, $\rho$ lifts $\rho_1$ and $\rho_2$.
By Lemma \ref{surjectivity}, $\rho$ is surjective.
\end{proof}

We will apply Theorem \ref{main lifting theorem} in the following form:

\begin{cor}
\label{lifting for Vn}
Assume that $N=n+1$, 
\[
I=\{i_1\nek i_n\}, \quad {\widehat I}_1=\{i_1\nek i_{n-1},i_{n+1}\}, \quad {\widehat I}_2=\{i_2\nek i_n,i_{n+1}\},
\]
and suppose that $\calG_{{\widehat I},I}$ has a $\dbV_n$-link type.
Let $\rho_1\colon G_{{\widehat I}_1}\to\dbU_{n-1}$ and $\rho_2\colon G_{{\widehat I}_2}\to\dbU_{n-1}$ be globalizations for $\calG_{{\widehat I}_1,{\widehat I}_1\cap I}$, $\calG_{{\widehat I}_2,{\widehat I}_2\cap I}$, respectively, which are trivial on $\res_i(\bar Z_i)$ for all $i\in I$.
Then there is a surjective globalization $\rho\colon G_{{\widehat I}}\to\dbU_n$ for $\calG_{{\widehat I},I}$ which lifts both $\rho_1$ and $\rho_2$, and such that $\rho(\res_i(\bar Z_i))\in \dbV_n$ for all $i\in I$.
\end{cor}

\subsection{Linking invariants}
\label{subsection on linking invariants}
Take $j\in J\setminus {\widehat I}$.
Let $\rho$ be a globalization for $\calG_{{\widehat I},I}$ which maps the image of $Z_j$ in $G_{{\widehat I}}$ into $\dbV_n$, where $n=|I|$.
Recall that $\pr_{1,n+1}\colon\dbV_n\to R^+$ is an isomorphism (Example \ref{Sr}(2)).

\begin{defin}
\label{linking invariant}
\rm 
The \textsl{linking invariant} of $({\widehat I},I)$ at $j$ is the pro-$p$ group homomorphism
\[
[{\widehat I},I,j]:=\pr_{1,n+1}\circ\rho\circ\res_j\colon Z_j\to R^+,
\]
where $\res_j\colon Z_j\to G_{{\widehat I}}$ is the restriction homomorphism.
\end{defin}
Since the image of $T_j$ in $G_{{\widehat I}}$ is trivial, $[{\widehat I},I,j]$ factors via a pro-$p$ group homomorphism
\[
\pr_{1,n+1}\circ\rho\circ\res_j\colon \bar Z_j=Z_j/T_j\to R^+,
\]
and whenever convenient we identify these two maps.

\section{Linking Invariants and Magnus Theory}
\label{section on linking invariants and Magnus theory}
Morishita's arithmetic Milnor invariants are defined using coefficients in the Magnus homomorphism (described below).
We now explain the connection between our linking invariant $[{\widehat I},I,j]$ and the Magnus homomorphism.

\subsection{Magnus theory}
Let $R$ be a pro-$p$ unital ring and let $X=\{x_j\}_{j\in J}$ be a set of variables.
For every word $I=(i_1\cdots i_n)$ in the alphabet $J$ we write $x_I=x_{i_1}\cdots x_{i_n}$.
The \textsl{Magnus algebra} $R\langle\langle X\rangle\rangle$ is the $R$-algebra of formal power series in the non-commuting variables $x_j$, $j\in J$, and with coefficients in $R$ (see \cite{Serre02}*{Ch.\ I, Sect.\ 1.4}).
We write its elements as $\sum_Ic_Ix_I$ with $c_I\in R$.
Let $R\langle\langle X\rangle\rangle^{\times,1}$ be its group of $1$-units, that is, series with $c_\emptyset=1$.
It is a pro-$p$ group.

Assume further that for every $j\in J$ we have an isomorphic copy $T_j$ of $\dbZ_p$ with a generator $\tau_j$.
Consider the free pro-$p$ product $\star_{j\in J}T_j$.
It is a free pro-$p$ group on the basis $\tau_j$, $j\in J$.
The \textsl{Magnus homomorphism} 
\[
\Lam\colon \star_{j\in J}T_j\to R\langle\langle X\rangle\rangle^{\times,1}
\]
is defined on the free generators by $\Lam(\tau_j)=1+x_j$.
For an arbitrary $\tau$ in this free product we write 
\[
\Lam(\tau)=\sum_I\eps_I(\tau)x_I,
\]
where $I$ ranges over all finite words in the alphabet $J$ and $\eps_I(\tau)\in R$.

\subsection{The linking invariant as a Magnus coefficient}
Suppose that $R$ is a pro-$p$-cyclic ring, i.e., $R=\dbZ_p$ or $R=\dbZ/q$ for a $p$-power $q$.
Let $G$, $Z_j$, $T_j$, $j\in J$ be as in Section \ref{section on the structures GI}.
We assume that $T_j\isom\dbZ_p$ with a generator $\tau_j$ as before, so $\hat T=\star_{j\in J}T_j$ is a free pro-$p$ group.

Let $I=\{i_1\nek i_n\}$ be a subset of $J$ of size $n$, considered also as a word in the alphabet $J$, and let $\dbU_n=\dbU_n(R)$.
For $j\in J$ we take the epimorphism 
\[
\bar\pi_j\colon T_j\to R^+, \quad \tau_j\mapsto 1,
\]
and define $\pi_j\colon T_j\to\dbU_n$ as in (\ref{images of pi j}).

\begin{lem}
\label{Magnus matrix}
For every $j\in J$ and $\tau\in T_j$ one has 
\[
\pi_j(\tau)=\Bigl[\eps_{(i_l\cdots i_{k-1})}(\tau)\Bigr]_{1\leq l\leq k\leq n+1}.
\]
\end{lem}
\begin{proof}
As $\Lam(\tau_j)=1+x_j$, for every $1\leq l\leq k\leq n+1$ we have 
\[
\eps_{(i_l\cdots i_{k-1})}(\tau_j)=\begin{cases}1,&\hbox{if }l=k,\\1,&\hbox{if }j=i_l,\ k=l+1,\\ 0,& \hbox{otherwise.}\end{cases}
\]
When compared to (\ref{images of pi j}), this gives the asserted equality for $\tau=\tau_j$.
By \cite{Efrat14}*{Lemma 7.5}, the right-hand side is a group homomorphism $\hat T\to\dbU_n$, so the equality holds for every $\tau\in T_j$.
\end{proof}

Let $\calG_{{\widehat I},I}$ be as in (\ref{G hat I I}) and let $\rho\colon G_{{\widehat I}}\to\dbU_n$ be a globalization for $\calG_{{\widehat I},I}$.
Then $\rho\circ\res_i=\pi_i$ on $T_i$ for $i\in I$, so by Lemma \ref{Magnus matrix}, $\pr_{1,n+1}\circ\rho\circ\res_i=\eps_I$ on $T_i$.
For the free product epimorphism $\hat\lam_T=\star_{i\in I}\res_i\colon \hat T=\star_{i\in I}T_i\to G_{{\widehat I}}$ (see Subsection \ref{subsection on lifting of globalizations in link type structures}), we therefore have on $\hat T$:
\begin{equation}
\label{eps and globalizations}
\pr_{1,n+1}\circ\rho\circ\hat\lam_T=\eps_I
\end{equation}

\begin{cor}
Let $j\in J\setminus {\widehat I}$, $\sig\in Z_j$, and $\hat\sig\in\hat Z_j$ satisfy $\res_j(\sig)=\hat\lam_T(\hat\sig)$.
Then 
\[
[{\widehat I},I,j](\sig)=\eps_I(\hat\sig).
\]
\end{cor}
\begin{proof}
We compute using (\ref{eps and globalizations}):
\[
[{\widehat I},I,j](\sig)=(\pr_{1,n+1}\circ\rho\circ\res_j)(\sig)=(\pr_{1,n+1}\circ\rho\circ\hat\lam_T)(\hat\sig)=\eps_I(\hat\sig).
\qedhere
\]
\end{proof}

\section{Valuations and Orderings}
\label{section on valuations and orderings}
In this section, we recall some basic notions and facts on valuations and orderings on fields.
We refer e.g.\ to \cite{Efrat06} for more details and for proofs.

\subsection{Preliminaries on valuations}
For an abelian group $A$ and a positive integer $m$ we write ${}_mA$ and $A/m$ for the kernel and cokernel, respectively, of the map $A\xrightarrow{m}A$ of multiplication by $m$.

Let $F$ be a field, let $F^\times$ be its multiplicative group, and let $\mu(F)$ be the group of all roots of unity in $F$.
We write $\mu_m(F)$ for the group ${}_m(F^\times)$ of all roots of unity in $F$ of order dividing $m$.
For a prime number $p$ let $F(p)$ be the maximal pro-$p$ Galois extension of $F$, and $G_F(p)=\Gal(F(p)/F)$ the maximal pro-$p$ Galois group of $F$.

Recall that a (Krull) valuation on $F$ is a group epimorphism $v\colon F^\times\to\Gam_v$, where $\Gam_v=(\Gam_v,\geq)$ is an ordered abelian group, satisfying the \textsl{ultra-metric inequality} $v(a+b)\geq\min\{v(a),v(b)\}$ for every $a,b\in F^\times$ with $a\neq-b$.
Equivalently, 
\[
O_v=\{a\in F^\times\ |\ v(a)\geq0\}\cup\{0\}
\]
is a valuation ring in $F$.
Its unique maximal ideal is 
\[
\grm_v=\{a\in F^\times\ |\ v(a)>0\}\cup\{0\},
\]
and its residue field is $\bar F_v=O_v/\grm_v$.
We denote the residue of $a\in O_v$ in $\bar F_v$ by $\bar a$. 
Being an ordered abelian group, $\Gam_v$ is torsion-free.
The valuation $v$ is called \textsl{discrete} if $\Gam_v$ is isomorphic as an ordered abelian group to $\dbZ$ with its natural order.
Then an element $t$ of $F^\times$  such that $v(t)$ corresponds to $1\in\dbZ$ is called a \textsl{uniformizer} for $v$.

Valuations $v,v'$ on $F$ are called \textsl{equivalent} if $O_v=O_{v'}$;
Alternatively, there is an isomorphism $\varphi\colon \Gam_v\xrightarrow{\sim}\Gam_{v'}$ of ordered abelian groups such that $v'=\varphi\circ v$ \cite{Efrat06}*{Prop.\ 3.2.3(e)}.
We do not distinguish between equivalent valuations.

For every field extension $E$ of $F$ there is a valuation $u$ on $E$ which extends $v$ \cite{Efrat06}*{Cor.\ 14.1.2}.
Then there is an embedding $\Gam_v\hookrightarrow\Gam_u$ of ordered abelian groups, and a field embedding $\bar F_v\subseteq\bar E_u$.
When $E/F$ is algebraic, $\Gam_u/\Gam_v$ is a torsion group \cite{Efrat06}*{Cor.\ 14.2.3(a)}.

We now fix a positive integer $m$.

\begin{lem}
\label{snake lemma isomorphism}
${}_m(\Gam_u/\Gam_v)\isom\Ker(\Gam_v/m\to\Gam_u/m)$.
\end{lem}
\begin{proof} 
The maps of multiplication by $m$ give a commutative diagram with exact rows
\[
\xymatrix{
0\ar[r]&\Gam_v\ar[r]\ar[d]^{m}&\Gam_u\ar[r]\ar@{>->}[d]^{m}&\Gam_u/\Gam_v\ar[r]\ar[d]^{m}&0\\
0\ar[r]&\Gam_v\ar[r]&\Gam_u\ar[r]&\Gam_u/\Gam_v\ar[r]&\ 0.\\
}
\]
The middle vertical map is injective since $\Gam_u$ is torsion-free.
We now apply the snake lemma (Cf.\ \cite{Efrat06}*{Lemma 1.1.4(a)}).
\end{proof}

\begin{rem}
\label{explicit snake lemma isomorphism}
\rm
More explicitly, take $\gam\in\Gam_u$ with $\gam+\Gam_v\in{}_m(\Gam_u/\Gam_v)$.
Then $m\gam\in\Gam_v\cap m\Gam_u$, and the above isomorphism maps $\gam+\Gam_v$ to $m\gam+m\Gam_v$.
\end{rem}

Now let $E$ be a Galois extension of $F$ and let $G=\Gal(E/F)$.
Recall that the decomposition, inertia, and ramification groups of the valued field extension $(E,u)/(F,v)$ are defined by
\[
\begin{split}
Z=Z(u/v)&=\{\sig\in\Gal(E/F)\ |\ u\circ\sig=u\},\\
T=T(u/v)&=\{\sig\in\Gal(E/F)\ |\ \forall a\in O_u:\ u(\sig(a)-a)>0\},\\
V=V(u/v)&=\{\sig\in T(u/v)\ |\ \forall a\in E^\times:\ u(1-\sig(a)/a)>0\},
\end{split}
\]
respectively  \cite{Efrat06}*{Ch.\ 15--16}.
One has $Z\trianglerighteq T\trianglerighteq V$.
The quotient $T/V$ is an abelian profinite group \cite{Efrat06}*{Cor.\ 16.2.7(d)}.
The group $V$ is pro-$p'$, if $p'=\Char\bar F_v>0$, and is trivial if $\Char\bar F_v=0$ \cite{Efrat06}*{Th.\ 16.2.3}.
The quotient $Z/T$ is canonically isomorphic to $\Aut(\bar E_u/\bar F_v)$ \cite{Efrat06}*{Prop.\ 16.1.3(a)(c) and Cor.\ 15.2.3(b)}.

Any other extension of $v$ to $E$ is equivalent to $u\circ\sig$ for some $\sig\in G$ \cite{Efrat06}*{Th.\ 14.3.2}.
Then $Z(u\circ\sig/v)=\sig\inv Z(u/v)\sig$ and $T(u\circ\sig/v)=\sig\inv T(u/v)\sig$ \cite{Efrat06}*{Remarks 15.1.1(b) and 16.1.2(b)}.

\subsection{The tame discrete pro-$p$ case}
\label{subsection on the tame pro-p discrete case}
We now examine the case where $F$ contains a root of unity of order $p$, $E=F(p)$ for $p$ prime, and $\Char\bar F_v\neq p$.
Then $V=1$ and $Z=T\rtimes\bar Z$ for $\bar Z=Z/T\isom G_{\bar F_v}(p)$ \cite{Efrat06}*{Th.\ 16.1.1 and Th.\ 22.1.1}.
Following Iwasawa \cite{Iwasawa55} in the case of local fields, we describe $Z$ under the additional assumption that $v$ is discrete.
For simplicity, we may replace $F$ by the fixed field of $Z$, to assume that $G_F(p)=Z$.
Thus $v$ has a unique extension $u$ to $E=F(p)$, and $\Gam_u=\tfrac1{p^\infty}\Gam_v$ and $\overline{(F(p))}_u=\bar F_v(p)$ \cite{Efrat97}*{Lemma 1.1(b)}.

We have $T=T(u/v)\isom\dbZ_p$ \cite{Efrat06}*{Cor.\ 16.2.7(c) and Prop.\ 22.1.3}.
Denote the fixed field of $T$ by $F^T$.
We also choose a uniformizer $t\in F$ for $v$ and a compatible system $\root{p^r}\of t$ of $p^r$-roots of $t$, where $r\geq1$.
Let 
\[
F^{\bar Z}=F(\root{p^r}\of t\ |\ r\geq1), \quad \bar Z=\Gal(F(p)/F^{\bar Z}).
\]

\begin{prop}
\label{Iwasawa}
In this setup, one has:
\begin{enumerate}
\item[(a)]
$G=T\rtimes\bar Z$.
\item[(b)]
If $A\subseteq F^T\cap O_u^\times$ and the set $\bar A$ of residues of $A$ generates $\bar F_v(p)$ over $\bar F_v$, then  $F(A)=F^T$.
\item[(c)]
The conditions of (b) hold if $\bar F_v$ is finite and $A=\mu_{p^\infty}=\bigcup_{n\geq1}\mu_{p^n}(F(p))$.
\end{enumerate}
\end{prop}
\begin{proof}
(a) \quad
For every subextension $F\subseteq K\subseteq F(p)$ let $u_K$ be the unique extension of $v$ to $K$.
For subextensions $F\subseteq K\subseteq K'\subseteq F(p)$ one has, by Ostrowski's theorem \cite{Efrat06}*{Th.\ 17.2.1},
\begin{equation}
\label{Ostrowski}
[K':K]=(\Gam_{u_{K'}}:\Gam_{u_K})[\overline{(K')}_{u_{K'}}:\overline{K}_{u^K}]
\end{equation}
as supernatural numbers (see \cite{Efrat06}*{Ch.\ 13}).

By \cite{Efrat06}*{Prop.\ 16.1.3(e)}, $\Gam_v=\Gam_{u^T}$.
Since the value group of $F(\root {p^r}\of t)$ is $\frac1{p^r}\Gam_v$, (\ref{Ostrowski}) implies that its residue field is $\bar F_v$, and hence, the residue field of $F^{\bar Z}$ is $\bar F_v$.
Consequently, for $K=F^T\cap F^{\bar Z}$ we have $\Gam_{u^K}=\Gam_v$ and $\bar K_{u^K}=\bar F_v$. 
It follows from (\ref{Ostrowski}) that $K=F$.  

Similarly, $\Gam_{u^{\bar Z}}=\Gam_u$ and $\overline{(F^T)}_{u^T}=\overline{(F(p))}_u$, by \cite{Efrat06}*{Prop.\ 16.1.3(a)} and since the residue characteristic is $\neq p$.
Therefore the compositum $K=F^TF^{\bar Z}$ satisfies $\Gam_{u^K}=\Gam_u$ and $\bar K_{u^K}=\overline{(F(p))}_u$. 
By (\ref{Ostrowski}) again, $K=F(p)$.  

The latter two facts imply (a).

\medskip

(b) \quad
Denote $L=F(A)\subseteq F^T$.
From $\Gam_v=\Gam_{u^T}$ and $\overline{(F^T)}_{u^T}=\overline{(F(p))}_u=\bar F_v(p)$, we obtain that
\[
\Gam_v=\Gam_{u^L}=\Gam_{u^T}, \quad \bar F_v(\bar A)=\overline{L}_{u^L}=\overline{(F^T)}_{u^T}=\bar F_v(p).
\]
By (\ref{Ostrowski}) once more, $L=F^T$.

\medskip

(c) \quad
We only need to show that $\mu_{p^\infty}\subseteq F^T$.
Indeed, let $\zeta\in\mu_{p^\infty}$ and $\sig\in T$.
Then $\zeta\in O_u^\times$ and $\sig(\zeta)/\zeta\in \mu_{p^\infty}$.
By the definition of $T$ we have $\overline{\sig(\zeta)/\zeta}=1$.
But the residue map is injective on $\mu_{p^\infty}$, so $\sig(\zeta)/\zeta=1$, as desired.
\end{proof}

We refer to \cite{Efrat06}*{Sect.\ 22.1} for a description of the action of $\bar Z$ on $T$.

\subsection{Orderings}
When $p=2$ we will also be interested in \textsl{orderings} on $F$, by which we mean additively closed subgroups $P$ of $F^\times$ of index $2$ \cite{Efrat06}*{Ch.\ 6}.
Note that this implies that $\Char\,F=0$.
The elements of $P$ are considered as ``positive''.
It is \textsl{Archimedean} if for every $a\in F$ there is a positive integer $k$ such that $k\cdot1\pm a\in P$. 

A \textsl{relative real closure} of $F$ with respect to $P$ inside $F(2)$ (also called a \textsl{Euclidean closure} of $F$ at $P$) is a maximal extension $F_P$ of $F$ inside $F(2)$ such that there is an ordering $Q$ on $F_P$ with $P=Q\cap F$.
By Becker's relative version of the Artin-Schreier theory (\cite{Becker74}, \cite{Efrat06}*{Th.\ 19.2.10}), this is equivalent to $[F(2):F_P]=2$.
Furthermore, $Q=(F_P^\times)^2$. 
We choose such a closure $F_P$ and set $Z_P=T_P=\Gal(F(2)/F_P)\isom\dbZ/2$. 
Furthermore, any two relative real closures of $F$ with respect to $P$ are isomorphic over $F$ \cite{Efrat06}*{Th.\ 19.4.2}, which means that $Z_P$ and $T_P$ are determined up to conjugacy in $G_F(2)$.

We say that a finite list of non-trivial valuations and orderings on $F$ is \textsl{independent}, if they induce distinct topologies on $F$  \cite{Efrat06}*{Sect.\ 10.1}.
This is automatic for discrete (or more generally, rank $1$) valuations and Archimedean orderings \cite{Efrat06}*{Example 10.1.1 and Prop.\ 7.3.1}.

For a pro-$p$ Galois extension $F'$ of $F$, we say that a valuation $v$ with an extension $u$ to $F(p)$ a above, resp., an ordering $P$ (if $p=2$), on $F$ is \textsl{unramified} in $F'$ if $T(u/v)$, resp., $T_P$, fixes $F'$.
Note that this does not depend on the choice of $u$ or $F_P$.

\section{The Inertia Group and the Kummer Pairing}
\label{section on inertia group and the Kummer pairing}
In this section, we fix a positive integer $m$.
For a profinite group $S$ let $S^{(2)}=S^m[S,S]$ be the second term in its decreasing $m$-central filtration, and set $S^{[2]}=S/S^{(2)}$. 

Assume that $F$ is a field that contains a root of unity of order $m$.
Let $E/F$ be a Galois extension with $G=\Gal(E/F)$.
We denote by $F^\times\cap E^m$ the group of all $a\in F^\times$ such that $\root m\of a\in E$ (this is independent of the choice of the $m$th root).
Recall that, for $a\in F^\times\cap E^m$, the \textsl{Kummer homomorphism} is the map
\[
\kappa_a\colon G\to\mu_m(F), \quad \sig \mapsto\dfrac{\sig(\root m\of a)}{\root m\of a}.
\]
The \textsl{Kummer pairing} is the well-defined perfect bilinear map of abelian groups of exponent dividing $m$:
\[
G^{[2]}\times (F^\times\cap E^m)/(F^\times)^m\to\mu_m(F), \quad (\bar \sig,a(F^\times)^m)=\kappa_a(\sig).
\]
This pairing is functorial in $E$ in the natural sense.

Let again $u$ be a valuation on $E$ extending a valuation $v$ on $F$, and let $Z,T,V$ be as in Section \ref{section on valuations and orderings}.
We write $E_V$ for the fixed field of $V$ in $E$ and let $u_V$ be the restriction of $u$ to $E_V$.
There is a well-defined perfect bilinear map of abelian groups
\begin{equation}
\label{duality for inertia group}
T/V\times \Gam_{u_V}/\Gam_v\ \to\  \mu(\bar E_u), \quad
(\tau V, u_V(\alp)+\Gam_v)\mapsto \overline{\tau(\alp)/\alp}
\end{equation}
for $\tau\in T$ and $\alp\in E_V^\times$ \cite{Efrat06}*{Th.\ 16.2.6}.
Identifying $\mu(\bar E_u)$ as a subgroup of $\dbQ/\dbZ$, we therefore deduce that the abelian profinite group $T/V$ is the Pontrjagin dual of the torsion abelian (discrete) group $\Gam_{u_V}/\Gam_v$.
By the Pontrjagin duality theory, this pairing induces a duality between the cokernel of the  map $T/V\to T/V$ of exponentiation by $m$ and the kernel of the multiplication by $m$ map on $\Gam_{u_V}/\Gam_v$;
That is, there is a perfect bilinear map
\[
T/VT^m\times {}_m(\Gam_{u_V}/\Gam_v)\to\mu_m(\bar E_u)=\mu_m(\bar F_v), \  (\tau VT^m,u_V(\alp)+\Gam_v)\mapsto\overline{\tau(\alp)/\alp}.
\]
Combined with Lemma \ref{snake lemma isomorphism}, it yields a perfect bilinear map
\begin{equation}
\label{induced pairing}
T/VT^m\times\Ker(\Gam_v/m\to\Gam_{u_V}/m)\to\mu_m(\bar F_v).
\end{equation}

The valuation $v$ induces an abelian group homomorphism 
\[
\bar v\colon (F^\times\cap E^m)/(F^\times)^m\to\Ker(\Gam_v/m\to\Gam_u/m).
\]

\begin{prop}
\label{Kummer commutes with inertia pairing}
Let $E_1$ be a Galois extension of $F$ contained in $E_V$, and let $G_1=\Gal(E_1/F)$.
The pairing (\ref{induced pairing}) and the Kummer pairing combine to a commutative diagram of perfect bilinear maps
\[
\xymatrix{
T/VT^m\ar[d]&*-<3pc>{\times}& \Ker(\Gam_v/m\to\Gam_{u_V}/m)\ar[r]&\mu_m(\bar F_v) \\
G_1^{[2]}&*-<3pc>{\times}& (F^\times\cap E_1^m)/(F^\times)^m\ar[u]^{\bar v}\ar[r]&\mu_m(F)\ar[u]_{\wr},
}
\]
where the left vertical homomorphism is induced by restriction and the right vertical isomorphism is the residue map.
\end{prop}
\begin{proof}
Take $\tau\in T$ and $a\in F^\times\cap E_1^m$.
Let $\gam=u_V(\root m\of a)\in\Gam_{u_V}$.
It is independent of the choice of the root, and $\gam+\Gam_v\in{}_m(\Gam_{u_V}/\Gam_v)$.
By Remark \ref{explicit snake lemma isomorphism}, the isomorphism of Lemma \ref{snake lemma isomorphism} maps $\gam+\Gam_v$ to 
\[
\bar v(a(F^\times)^m)=m\gam+m\Gam_v\in\Ker(\Gam_v/m\to\Gam_{u_V}/m).
\]
Therefore in the upper pairing we have
\[
(\tau VT^m,\bar v(a(F^\times)^m))=\overline{\tau(\root m\of a)/\root m\of a}=\overline{\kappa_a(\tau)},
\] 
whereas in the lower (Kummer) pairing we have
\[
(\tau G_1^{(2)},a(F^\times)^m)=\kappa_a(\tau).
\qedhere
\]
\end{proof}

Now assume that $p$ is a prime number and $m=q>1$ is a $p$-power.

\begin{prop}
\label{Kummer map and valuations}
Suppose that $E=F(p)$ and $\Char\,\bar F_v\neq p$.
Let $a\in F^\times$.
\begin{enumerate}
\item[(a)]
If $v(a)\equiv 0\pmod {q\Gam_v}$, then $\kappa_a$ is trivial on $T$.
\item[(b)]
If $\bar v(a)$ generates $\Gam_v/q$ and $\Gam_v/q\isom\dbZ/q$, then $\kappa_a\colon T\to\mu_q(F)$ is surjective.
\end{enumerate}
\end{prop}
\begin{proof}
Since $\Char\,\bar F_v\neq p$ and $G=G_F(p)$ is pro-$p$, the ramification group $V$ is trivial, so $u_V=u$.
Since $E=E^p$ we have $\Gam_u/q=0$.
Proposition \ref{Kummer commutes with inertia pairing} therefore gives a commutative diagram of perfect bilinear maps
\begin{equation}
\label{cd of combined pairings}
\xymatrix{
T/T^q\ar[d]&*-<3pc>{\times}&\Gam_{v}/q\ar[r]&\mu_q(\bar F_v)\\
\Gal(F(\root q\of a)/F)&*-<3pc>{\times}&(F^\times\cap (F(\root q\of a))^q)/(F^\times)^q\ar[u]^{\bar v}\ar[r]&\mu_q(F).\ar[u]_{\wr}
}
\end{equation}

(a) \quad
By assumption, $\bar v(a)=0$, so for every $\tau\in T$ one has $(\tau T^q,\bar v(a))=0$ with respect to the upper pairing in (\ref{cd of combined pairings}). 
Hence the image $\bar\tau$ of $\tau$ in $\Gal(F(\root q\of a)/F)$ satisfies $(\bar\tau,a)=0$ with respect to the lower (Kummer) pairing,
that is $\kappa_a(\tau)=1$.

\medskip

(b) \quad 
In (\ref{cd of combined pairings}), $\bar v$ is surjective, so the map $T/T^q\to\Gal(F(\root q\of a)/F)$ is injective.
Moreover, since the upper pairing in (\ref{cd of combined pairings}) is perfect, $T/T^q\isom\dbZ/q$.
We deduce that $\Gal(F(\root q\of a)/F)$ has order $q$, and $T/T^q\xrightarrow{\sim}\Gal(F(\root q\of a)/F)$ via restriction.
Since the lower pairing is perfect, $\kappa_a\colon\Gal(F(\root q\of a)/F)\to\mu_q(F)$ is surjective, and the assertion follows.
\end{proof}

The record the following analogous facts for orderings:

\begin{prop}
\label{Kummer map and orderings}
Suppose that $q=2$ and $E=F(2)$.
Let $P$ be an ordering on $F$ with relative real closure $F_P$, and let $T=\Gal(F(p)/F_P)$.
\begin{enumerate}
\item[(a)]
If $a\in P$, then $\kappa_a$ is trivial on $T$.
\item[(b)]
If $a\not\in P$, then $\kappa_a\colon T\to\mu_2(F)=\{\pm1\}$ is surjective.
\end{enumerate}
\end{prop}
\begin{proof}
The unique ordering on $F_P$ is $(F_P^\times)^2$.
Hence $P=F^\times\cap F_P^2$, from which both assertions follow.
\end{proof}

\section{Kummer Linking Structures}
\label{section on Kummer linking structures}
We now arrive at our main construction of linking structures for valuations and orderings.
It extends to general fields the notion of the maximal pro-$p$ Galois group of a number field unramified outside a finite set of primes and is based on the group-theoretic construction of the structures $\calG_{{\widehat I},I}$ in Section \ref{section on the structures GI}.

\subsection{Field-theoretic setup}
Let again $p$ be a prime number and $q>1$ a $p$-power.
Let $F$ be a field which contains a root of unity of order $q$.
By choosing such a root, we fix once and for all an isomorphism $\dbZ/q\isom \mu_q$.
Set as before $G=G_F(p)=\Gal(F(p)/F)$.

Let $v_j$, $j\in J_\val$, be a set of non-equivalent non-trivial valuations on $F$, and for each $j$ fix an extension $u_j$ of $v_j$ to $F(p)$.
Let $Z_j\geq T_j$ be the decomposition and inertia groups, respectively, of $u_j$ over $v_j$ in $G$, and let $\bar Z_j=Z_j/T_j\isom G_{\bar F_j}(p)$.

When $p=2$ let $P_j$, $j\in J_\ord$, be distinct orderings on $F$.
We then choose a relative real closure $F_{P_j}$ of $F$ in $F(2)$ with respect to $P_j$, and set $Z_j=T_j=\Gal(F(2)/F_{P_j})\isom\dbZ/2$ and $\bar Z_j=1$.
When $p\neq2$ take $J_\ord=\emptyset$.
Note that, as $F$ contains a $q$th root of unity, $J_\ord\neq\emptyset$ implies that $q=2$.

Set $J=J_\val\discup J_\ord$.

Now let $I=\{i_1\nek i_n\}$ be a subset of $J$ of $n$ elements and write $I_\val=I\cap J_\val$ and $I_\ord=I\cap J_\ord$.
We assume that for every $i\in I_\val$ the valuation $v_i$ is discrete and $\Char\,\bar F_{v_i}\neq p$.
We recall from Subsection \ref{subsection on the tame pro-p discrete case} that $T_i\isom\dbZ_p$ and $Z_i=T_i\rtimes\bar Z_i$.
The latter equality trivially holds also for $i\in I_\ord$.
Moreover, we assume that $v_i$, $i\in I_\val$, and $P_i$, $i\in I_\ord$, are independent.

\begin{lem}
\label{t i}
There exist elements $t_i\in F^\times$, $i\in I$, such that:
\begin{enumerate}
\item[(1)]
$t_i$ is a uniformizer for $v_i$, if $i\in I_\val$;
\item[(2)]
$t_i\in -P_i$, if $i\in I_\ord$;
\item[(iii)]
$t_i$ is a $v_{i'}$-unit whenever $i\neq i'\in I_\val$;
\item[(iv)]
$t_i\in P_{i'}$ whenever $i\neq i'\in I_\ord$. 
\end{enumerate}
\end{lem}
\begin{proof}
For $i\in I_\val$ let $U_i=O_{v_i}^\times$, and for $i\in I_\ord$ let $U_i=P_i$.
For every $a\in F^\times$ the set $aU_i$ is open in the $v_i$-topology (resp., $P_i$-topology) on $F$ \cite{Efrat06}*{Section 8.1}.

Now take $t'_i$ to be a uniformizer for $v_i$, if $i\in I_\val$, and take $t'_i=-1$ if $i\in I_\ord$. 
By the weak approximation theorem for independent valuations and orderings \cite{Efrat06}*{Th.\ 10.1.7 and Example 10.1.1}, we may choose
\[
t_i\in t'_iU_i\cap\bigcap_{i'\in I\setminus \{i\}}U_{i'}.
\qedhere
\] 
\end{proof}

We choose $t_i$, $i\in I$, as in Lemma \ref{t i}.
For every $i\in I$ take 
\[
\bar\pi_i=\kappa_{t_i}\colon T_i\to\mu_q\isom \dbZ/q
\]
to be the Kummer map of $t_i$.
Proposition \ref{Kummer map and valuations}(a) and Proposition \ref{Kummer map and orderings}(a) imply that it does not depend on the choice of $t_i$.
Further, by Proposition \ref{Kummer map and valuations}(b) and Proposition \ref{Kummer map and orderings}(b), $\bar\pi_i$ is surjective.

Define $\pi_i\colon T_i\to\dbU_n$ as in (\ref{images of pi j}), with $R=\dbZ/q$.

As before, consider an intermediate set $I\subseteq {\widehat I}\subseteq J$, and let $N_{{\widehat I}}$ and $G_{{\widehat I}}=G/N_{{\widehat I}}$ be as in Section \ref{section on the structures GI}.
In particular, we assume that (A) holds, that is, the images of $T_i$, $i\in I$, generate $G_{{\widehat I}}$.
This means that $F$ has no proper Galois extension inside $F(p)$ which is unramified with respect to all valuations and orderings in $I\cup(J\setminus{\widehat I})$.

\begin{rem}
\label{independence of extension}
\rm
The validity of assumption (A) in our context is independent of the choice of the extensions $u_j$ and the relative real closures $F_{P_j}$ (if $p=2$).
Indeed, as noted in Section \ref{section on valuations and orderings}, choosing another extension $u_j$ (resp., another relative real closure $F_{P_j}$) amounts to replacing $T_j$ by a conjugate subgroup of $G$.
A standard Frattini argument shows that the pro-$p$ group $G_{{\widehat I}}$ is generated by a family of closed subgroups if and only if it is generated by conjugates of these subgroups.
\end{rem}

In this setup we have a linking structure $\calG_{{\widehat I},I}$ as in (\ref{G hat I I}).

\begin{defin}
\label{Kummer linking structure}
\rm
We call $\calG_{{\widehat I},I}$ the $I$-indexed \textsl{Kummer linking structure of $F$ unramified outside $\widehat I$}.
\end{defin}

Globalizations for $\calG_{{\widehat I},I}$ do not depend on the choices of the extensions $u_j$ of the $v_j$, or the relative real closures with respect to the $P_j$, if $p=2$.
Indeed, as noted above, other choices give conjugates of $Z_j,T_j$, and we apply Remark \ref{conjugates}.

Moreover, globalizations for $\calG_{{\widehat I},I}$ are related to Kummer maps as follows:

\begin{prop}
\label{existence of a globalization}
Let $\rho\colon G_{{\widehat I}}\to\dbU_n$ be a globalization for $\calG_{{\widehat I},I}$.
For every $1\leq l\leq n$ one has $\pr_{l,l+1}\circ\rho=\kappa_{t_{i_l}}$ on $G_{{\widehat I}}$.
\end{prop}
\begin{proof}
If $i\in I$ and $i\neq i_l$, then $t_{i_l}$ is a $v_i$-unit, if $i\in I_\val$, and is in $P_i$, if $i\in I_\ord$.
Proposition \ref{Kummer map and valuations}(a) and Proposition \ref{Kummer map and orderings}(a) imply that $\kappa_{t_{i_l}}=0$ on $T_i$.
It follows from (\ref{images of pi j}) that for every $1\leq l\leq n$ and $i\in I$ we have $\pr_{l,l+1}\circ\pi_i=\kappa_{t_{i_l}}$ on $T_i$.
Hence
\[
\pr_{l,l+1}\circ\rho\circ\res_i=\kappa_{t_{i_l}}
\]
on $T_i$, where $\res_i\colon T_i\to G_{{\widehat I}}$ is the restriction map.
Finally, by assumption, the subgroups $\res_i(T_i)$, $i\in I$, generate $G_{{\widehat I}}$.
\end{proof}

\subsection{Ramification and decomposition}
\label{subsection on ramification and decomposition}
We now interpret this construction from a field-theoretic perspective.
Define $E_{{\widehat I}}$ to be the fixed field of $N_{{\widehat I}}$ in $F(p)$, so $G_{{\widehat I}}=\Gal(E_{{\widehat I}}/F)$.
The field $E_{{\widehat I}}$ is the maximal pro-$p$ Galois extension of $F$ which is unramified at all $v_j$ and $P_j$, $j\in J\setminus {\widehat I}$.
Given a globalization $\rho$ for $\calG_{{\widehat I},I}$, we write $E_\rho$ for the fixed field of $\Ker(\rho)$ in $F(p)$.
It is a Galois extension of $F$ with $F\subseteq E_\rho\subseteq E_{{\widehat I}}\subseteq F(p)$.
Since $\rho$ is onto $\dbU_n$ (Lemma \ref{surjectivity}), 
\[
\Gal(E_\rho/F)=G_{{\widehat I}}/\Ker(\rho)\isom\dbU_n.
\]

\begin{thm}
\label{ramification}
\begin{enumerate}
\item[(a)]
If $i\in I$, then the valuation, resp., ordering corresponding to $i$ ramifies in the extension $E_\rho/F$ and the image of $T_i$ in $\Gal(E_\rho/F)$ is isomorphic to $\dbZ/q$.
\item[(b)]
If $j\in J\setminus {\widehat I}$, then the valuation, resp., the ordering corresponding to $j$ is unramified in $E_\rho/F$.
\end{enumerate}
\end{thm}
\begin{proof}
(a) \quad
Write $i=i_l\in I$.
Then the image of $T_i$ in $\Gal(E_\rho/F)$ is isomorphic to 
\[
(\rho\circ\res_i)(T_i)=\pi_i(T_i)=\Id_{\dbU_n}+\Img(\bar\pi_i)E_{l,l+1}\isom\Img(\bar\pi_i)=\mu_q\isom\dbZ/q,
\]
by the definition of $\pi_i$ (\ref{images of pi j}) and Lemma \ref{elementary matrices}(a).

\medskip

(b) \quad
As $j\in J\setminus {\widehat I}$, the image of $T_j$  is trivial already in $G_{{\widehat I}}$, whence also in $\Gal(E_\rho/F)$.
\end{proof}

For a globalization $\rho$ for $\calG_{{\widehat I},I}$ and for $j\in J_\val\setminus {\widehat I}$ such that $\rho(\res_j(Z_j))\subseteq \dbV_n$, the linking invariant $[{\widehat I},I,j]$ (Definition \ref{linking invariant}) has the following arithmetical interpretation:  

\begin{thm}
\label{vanishing of linking invariant}
Under the above assumptions, one has $[{\widehat I},I,j]=0$ if and only if the valuation $v_j$ is completely decomposed in $E_\rho$.
\end{thm}
\begin{proof} 
By its definition, $[{\widehat I},I,j]$ vanishes if and only if $\rho(\res_j(Z_j))$ (which by assumption is contained in $\dbV_n$) is trivial.
Equivalently, the restriction of $Z_j$ to $E_\rho$ is trivial, i.e., $v_j$ completely decomposes in $E_\rho$.
\end{proof}

\begin{rem}
\label{linking invariant for orderings}
\rm
The analog of Theorem \ref{vanishing of linking invariant} in the case $j\in J_\ord\setminus {\widehat I}$, $p=2$, also holds in the following sense:
Since $Z_j=T_j$, we have $[{\widehat I},I,j]=0$.
As $j\in J\setminus {\widehat I}$, Theorem \ref{ramification}(b) implies that $P_j$ extends to $E_\rho$.  
\end{rem}

\begin{defin}
\label{linking homomorphism}
\rm 
Let $i=i_l\in I_\val$, let $j\in J_\val$ satisfy $\Char\,\bar F_{v_j}\neq p$, and let $\rho$ be a globalization for $\calG_{{\widehat I},I}$.
The \textsl{linking homomorphism} ${\rm lk}(v_i,v_j)$ of the valuations $v_i,v_j$ is the composed homomorphism
\[
{\rm lk}(v_i,v_j)\colon G_{\bar F_{v_j}}(p)\isom \bar Z_j\xrightarrow{\res_j}G_{{\widehat I}}\xrightarrow{\rho}\dbU_n\xrightarrow{\pr_{l,l+1}}\dbZ/q.
\]
Note that, by Example \ref{special projections are homomorphisms}(3), $\pr_{l,l+1}$ is a group homomorphism.
\end{defin}

\begin{rem}
\rm
(1) \quad
Using Remark \ref{uniqueness}, we see that ${\rm lk}(v_i,v_j)$ is independent of the choices of $\rho$ and $I,{\widehat I}$.

\medskip

(2) \quad
To define ${\rm lk}(v_i,v_j)$, we in fact only need to have a globalization for $\calG_{{\widehat I},I}/V_S$, where $S=\dbI_n\cup\{(l,l+1)\}$ (with notation as in Section \ref{section on unitriangular matrices}).

\medskip

(3)\quad
Similarly we may define the linking homomorphism ${\rm lk}(P_i,v_j)$ for $i\in I_\ord$ and $j$ as above.

\medskip

(4) \quad
As will be explained in Remark \ref{connection with Morishitas linking numbers} below, Definition \ref{linking homomorphism} essentially extends Morishita's notion of the \textsl{linking number} of finite primes in a number field, as in \cite{Morishita12}*{Th.\ 7.4}, to the context of discrete valuations on general fields.
\end{rem}

\section{Number Fields}
\label{section on number fields}
The main results in this and the next section are essentially due to Morishita and Amano (\cite{Morishita02}, \cite{Morishita04},  \cite{Amano14a}), and we present them here in the language of linking structures.
 
\subsection{Setup}
Fix again a $p$-power $q>1$.  
We restrict ourselves to the case where $F$ is a number field which contains a root of unity of order $q$.
Let $O_F$ be its ring of integers.

Let $v_j$, $j\in J_\val$, be the (non-equivalent) discrete valuations on $F$, and for each $j\in J_\val$ we choose a valuation $u_j$ on $F(p)$ extending $v_j$.
Let $\grq_j=\grm_{v_j}\cap O_F$ be the prime ideal in $O_F$ corresponding to $v_j$.
When $p=2$ let $P_j$, $j\in J_\ord$, be the (finitely many) orderings on $F$.
Set again $J=J_\val\discup J_\ord$, where $J_\ord=\emptyset$ if $p\neq2$.

Let $G=G_F(p)=\Gal(F(p)/F)$.
For $j\in J_\val$ we write $Z_j\geq T_j$ for the decomposition and inertia groups, respectively, of the extension $u_j/v_j$ in $G$, and $F^{Z_j}$, $F^{T_j}$ for the fixed fields of $Z_j$, $T_j$, respectively, in $F(p)$.
When $p=2$ and for $j\in J_\ord$ we fix as before a relative real closure $F_{P_j}$ of $F$ in $F(2)$ with respect to $P_j$, and take $Z_j=T_j=\Gal(F(2)/F_{P_j})$.

\subsection{Power residue symbols}
For each $j\in J_\val$ with $\Char\,\bar F_{v_j}\neq p$ we choose a uniformizer $t_j$ for $v_j$ and a compatible system $\root{p^r}\of{t_j}$, $r\geq1$, of $p^r$-roots of $t_j$.
By Proposition \ref{Iwasawa} (with $F$ replaced by $F^{Z_j}$),
$F^{T_j}=F^{Z_j}(\mu_{p^\infty})$ and $Z_j=T_j\rtimes\bar Z_j$, where 
\[
\bar Z_j=\Gal(F(p)/F^{Z_j}(\root{p^r}\of{t_j}\ | r\geq1)).
\]
Furthermore, the subgroup $\bar Z_j$ is mapped isomorphically onto $G_{\bar F_{v_j}}(p)$ by the epimorphism $Z_j\to Z_j/T_j\isom G_{\bar F_{v_j}}(p)$.
Let $\sig_j\in\bar Z_j$ be the lift of the Frobenius automorphism of $G_{\bar F_{v_j}}(p)$.

Since $F$ contains a root of unity of order $q$, so does the finite field $\bar F_{v_j}$.
Therefore 
\[
N_j:=|\bar F_{v_j}|\equiv1\pmod q.
\]

Recall that, when $\Char\,\bar F_{v_j}\neq p$ and for $\alp\in O_{v_j}^\times$, the \textsl{power residue symbol} $\left(\frac\alp{\grq_j}\right)_q$ is the unique $q$th root of unity $\zeta$ in $F$ such that
\[
\zeta\equiv\alp^{(N_j-1)/q}\pmod {\grq_j}
\]
\cite{Neukirch99}*{Ch.\ V, Section 3}.

\begin{lem}
\label{kappa and power residue symbols}
Let  $j\in J_\val$ such that $\Char\,\bar F_{v_j}\neq p$ and $t\in O_F\setminus \grq_j$.
Then 
\[
\kappa_{t}(\sig_j)=\Biggl(\dfrac{t}{\grq_j}\Biggr)_q.
\]
\end{lem}
\begin{proof}
As $\root q\of t\in O_{u_j}^\times$, and by the choice of $\sig_j$, 
\[
\kappa_t(\sig_j)=\frac{\sig_j(\root q\of t)}{\root q\of{t_i}}\equiv t^{(N_j-1)/q}\pmod{\grm_{u_j}}.
\]
Both sides of this congruence are in $O_F$, so it holds modulo $\grq_j$.
Therefore $\kappa_t(\sig_j)=\bigl(t/\grq_j\bigr)_q$.
\end{proof}

Consider as before finite subsets $I\subseteq {\widehat I}$ of $J$.
We write 
\[
I=\{i_1\nek i_n\}, \quad {\widehat I}_\val={\widehat I}\cap J_\val, \quad {\widehat I}_\ord={\widehat I}\cap J_\ord.
\]
Assume that $\Char\,\bar F_{v_i}\neq p$ for every $i\in \widehat I_\val$, 
and take $t_i\in F^\times$, $i\in \widehat I$, as in Lemma \ref{t i}.
We make as before assumption (A) of Section \ref{section on the structures GI}, namely, that $G_{{\widehat I}}$ is generated by the images of $T_i$, $i\in I$.
Equivalently, $F$ does not have a proper $p$-extension which is unramified at all places in $(J\setminus {\widehat I})\cup I$ (with notation as in subsection 6.3).
Let $\calG_{{\widehat I},I}$ be the linking structure as in Definition \ref{Kummer linking structure}.

For every $i\in \widehat I$ we fix a pro-$p$ generator $\tau_i$ of $T_i$ (which is isomorphic to $\dbZ_p$, if $i\in \widehat I_\val$, or to $\dbZ/2$, if $i\in \widehat I_\ord$ and $p=2$). %
In view of Propositions \ref{Kummer map and valuations}(b) and \ref{Kummer map and orderings}(b), we may assume that $\kappa_{t_i}(\tau_i)=1\pmod{q\dbZ}$ under the fixed isomorphism $\mu_q=\dbZ/q$.
Write $\bar\sig_i,\bar\tau_i$ for the images of $\sig_i,\tau_i$ in $G_{{\widehat I}}$.

\begin{lem}
\label{power residue symbols}
Let $\rho\colon G_{{\widehat I}}\to \dbU_n$ be a globalization for $\calG_{{\widehat I},I}$.
Also let $i=i_l\in I$ and $j\in I_\val$.
Then
\[
(\pr_{l,l+1}\circ\rho)(\bar\tau_j)=\delta_{ij}.
\]
\end{lem}
\begin{proof}
We have $\pr_{l,l+1}\circ\rho\circ\res_j=\pr_{l,l+1}\circ\pi_j$ on $T_j$.
By (\ref{images of pi j}), this is $0$, if $j\neq i$, and is $\kappa_{t_i}$, if $j=i$.
In the latter case, the assertion follows from the choice of $\tau_i$.
\end{proof}

\subsection{The Koch--Hoechsmann presentation}
Given the finite subset $\widehat I$ of $J$, the structure of $G_{{\widehat I}}$ was computed by Koch (\cite{Koch95}*{Th.\ 11.10}, \cite{Haberland78}*{Appendix 1}) and Hoechsmann \cite{Hoechsmann66} under certain assumptions arising from class field theory; see also \cite{Mizusawa19}.
To state this result, we write $\Div(F)$ for the divisor group of $F$, and ${\rm div}\colon F^\times\to\Div(F)$ for the principal divisor homomorphism.
Define
\[
\CyrB_{{\widehat I}}=\Bigl({\rm div}\inv(p\Div(F))\cap\bigcap_{i\in {\widehat I}}(F^{Z_i})^p\Bigr)/(F^\times)^p.
\]
Recall that we are assuming that $\mu_p\subseteq F$.
 
\begin{thm}[Koch, Hoechsmann]
\label{KochHoechsmann}
Suppose that $q=p$, ${}_p\Cl(F)=0$, and $\CyrB_{{\widehat I}}=1$.
Then there is a subset $I$ of $\widehat I$ such that the pro-$p$ group $G_{{\widehat I}}$ is minimally generated by $\bar\tau_i$, $i\in I$, subject to the defining relations 
\begin{enumerate}
\item[(i)]
$\bar\tau_i^{N_i-1}[\bar\tau_i\inv,\bar\sig_i\inv]=1$, if $i\in {\widehat I}_\val$;
\item[(ii)]
$\bar\tau_i^2=1$, if $i\in {\widehat I}_\ord$, $p=2$.
\end{enumerate} 
Here $\bar\tau_i$, $\bar\sig_i$, for $i\in \widehat I\setminus I$, are expressed as pro-$p$ words in the generators $\bar\tau_i$, $i\in I$.

Furthermore, we may omit any of these relations.
\end{thm}

See \cite{Koch95}*{\S 11.4} for a detailed description of the subset $I$.

In view of Example \ref{example of link type group}, we obtain:

\begin{cor}
\label{link type for number fields}
Under the assumptions of Theorem \ref{KochHoechsmann}, 
the linking structure $\calG_{{\widehat I},I}$ has a $\dbV_n$-link type.
\end{cor}

\begin{rem}
\rm
The group $\CyrB_{{\widehat I}}$ is defined in \cite{Koch95} in terms of the completions $\hat F_i$ of $F$ with respect to $v_i$ or $P_i$, rather than the decomposition fields $F^{Z_i}$.
However, this does not matter, since  $((F^{Z_i})^\times)^p=(F^{Z_i})^\times\cap (\hat F_i^\times)^p$ via some choice of an embedding $F^{Z_i}\hookrightarrow \hat F_i$.
Indeed, for the non-obvious inclusion in the valuation case, we recall that the value groups and residue fields of $\hat F_i$ and $F^{Z_i}$ coincide \cite{Efrat06}*{Th.\ 9.3.2(f) and Th.\ 15.2.2}.
The fundamental equality for discrete valued fields \cite{Efrat06}*{Th.\ 17.4.3} therefore implies that there is no extension of $F^{Z_i}$ of degree $p$ inside $\hat F_i$.
In the ordering case we have $[F(2):F^{Z_i}]=2$, $\sqrt{-1}\in F(2)$ and $\hat F_i\isom\dbR$, so we are done again.
\end{rem}

\begin{rem}
\label{connection with Morishitas linking numbers}
\rm
The above defining relations of $G_{{\widehat I}}$ are contained in $G_{{\widehat I}}^{(2)}=G_{{\widehat I}}^q[G_{{\widehat I}},G_{{\widehat I}}]$.
Hence, $G_{{\widehat I}}^{[2]}=G_{{\widehat I}}/G_{{\widehat I}}^{(2)}$ is a free $\dbZ/q$-module on the images of $\tau_i$, $i\in I$.
Therefore, for every $j\in J_\val$ with $\Char\,\bar F_{v_j}\neq p$ there is a unique factorization in $G_{{\widehat I}}^{[2]}$, 
\[
\sig_j\equiv\prod_{i\in I}\tau_i^{\mu(i,j)}\pmod{G_{{\widehat I}}^q[G_{{\widehat I}},G_{{\widehat I}}]}
\]
with $0\leq \mu(i,j)<q$. 
For $F=\dbQ$ and under certain assumptions on the primes $q_i$, $i\in I$, Morishita \cite{Morishita12}*{Th.\ 7.4} defines $\mu(i,j)$ to be the \textsl{mod-$q$ linking number} of the primes $\grq_i$ and $\grq_j$.

For $i=i_l$ and $j\in I$, 
Lemma \ref{power residue symbols} implies that $(\pr_{l,l+1}\circ\rho)(\bar\sig_j)=\mu(i,j)$ in $\dbZ/q$.
Since $\bar\sig_j$ generates $\bar Z_j$, the linking homomorphism ${\rm lk}(v_i,v_j)\colon\bar Z_j\to\dbZ/q$  (Definition \ref{linking homomorphism}), is therefore uniquely determined by $\mu(i,j)$.
In this sense, Definition \ref{linking homomorphism} extends Morishita's notion to the valuation-theoretic context.
\end{rem}

\section{Classical Invariants over $\dbQ$}
\label{section on classical invariants over Q}
We now focus on the case where $F=\dbQ$ and $q=2$.
We keep the notation as in the previous section.

\subsection{Globalizations for $\calG_{\{i,\infty\},\{i\}}$}
We write $q_j$, $j\in J_\val$, for the rational primes, so that $v_j$ is the $q_j$-adic valuation, $\grq_j=q_j\dbZ$, and we may take $t_j=q_j$.
If $j=\infty\in J_\ord$ corresponds to the unique ordering of $\dbQ$, then we take $t_j=-1$.
The power residue symbol $\bigl(\frac\alp{\grq_j}\bigr)_2$ coincides with the Legendre symbol $\bigl(\frac\alp {q_j}\bigr)$ for $\alp\in\dbZ\setminus q_j\dbZ$.
As $\Cl(\dbQ)=0$, assumption (A) holds. 
We recall the ramification behavior of rational primes in quadratic number fields \cite{IrelandRosen82}*{Ch.\ 13, \S1}:

\begin{prop}
\label{characterization of extensions by ramifications}
For a square-free integer $d$, a prime number $q_j$ ramifies in $\dbQ(\sqrt d)$ if and only if either $q_j|d$, or both $q_j=2$ and $d\not\equiv1\pmod4$.
\end{prop}

\begin{prop}
\label{isolated in 1 mod 4}
Let $i\in J_\val$ with $q_i\neq2$.
The following conditions are equivalent:
\begin{enumerate}
\item[(a)]
There is a globalization $\rho$ for $\calG_{\{i,\infty\},\{i\}}$;
\item[(b)]
$q_i\equiv1\pmod4$.
\end{enumerate}
Furthermore, in this case, $E_\rho=\dbQ(\sqrt{q_i})$.
\end{prop}
\begin{proof}
By Proposition \ref{equivalent conditions for isolated}, (a) means that there is a pro-$p$ group epimorphism $\rho_0\colon G=G_\dbQ(2)\to \dbZ/2$ which is $\kappa_{q_i}$ on $T_i$, and is trivial on all $T_j$, $j\in J_\val$, $j\neq i$.
Since  $\kappa_{q_i}(T_i)=\mu_2=\dbZ/2$ (Proposition \ref{Kummer map and valuations}(b)), this means that there is a square-free integer $d\neq1$ such that
\begin{enumerate}
\item[(i)]
$q_i$ ramifies in $\dbQ(\sqrt d)$;
\item[(ii)]
$q_j$ is unramified in $\dbQ(\sqrt d)$ for every $j\in J_\val$, $j\neq i$.
\end{enumerate}

When $q_i\equiv1\pmod4$, Proposition \ref{characterization of extensions by ramifications} implies that (i) and (ii) hold for $d=q_i$.

On the other hand, when $q_i\equiv3\pmod4$, Proposition \ref{characterization of extensions by ramifications} implies that there is no $d$ satisfying (i) and (ii).

For the last assertion, recall that $E_\rho$ is the fixed field of $\Ker(\rho)=\Ker(\rho_0)$, namely $\dbQ(\sqrt{q_i})$.
\end{proof}

Next let $n\geq1$ and let $i_1\nek i_n$ be distinct elements of $J$ such that $q_{i_1}\nek q_{i_n}$ are odd primes (that is, $v_{i_1}\nek v_{i_n}$ have residue characteristic $\neq2$).
Let $i_{n+1}=\infty$ correspond to the unique ordering on $\dbQ$ and set ${\widehat I}=\{i_1\nek i_n,i_{n+1}\}$.
Since $\dbQ$ does not have a proper unramified extension, assumption (A) of Section \ref{section on the structures GI} holds.
One has $\CyrB_{{\widehat I}}=1$ \cite{Koch95}*{Example 11.12}.
Further, in Theorem \ref{KochHoechsmann} we may take $I=\hat I_\val=\{i_1\nek i_n\}$ and omit the single relation corresponding to $i_{n+1}=\infty$  (\cite{Koch95}*{Example 11.12}, \cite{Morishita12}*{Th.\ 7.4}). 
We obtain a linking structure $\calG_{{\widehat I},I}$ with target group $\dbU_n$ as in (\ref{G hat I I}).
By Corollary \ref{link type for number fields}, $\calG_{{\widehat I},I}$ has $\dbV_n$-link type.

For $j\in J_\val\setminus {\widehat I}$, we now examine the associated linking invariants $[{\widehat I},I,j]$ (Definition \ref{linking invariant}) when $n$ is small.

\subsection{The case $n=1$ -- the Legendre symbol}
Let $i_1\in J_\val$, $i_2=\infty$, and assume that $q_{i_1}\equiv1\pmod4$.
Let $\rho$ be the unique globalization for $\calG_{\{i_1,i_2=\infty\},\{i_1\}}$ given by Proposition \ref{isolated in 1 mod 4} and Remark \ref{uniqueness}.
Then
\[
E_\rho=\dbQ(\sqrt{q_{i_1}}).
\]

Let $j\in J_\val\setminus \{i_1\}$.
The image of $[{\widehat I},I,j]$ in $\dbZ/2$ is generated by $\rho(\bar\sig_j)_{12}$.
Hence
\begin{equation}
\label{formula for Legendre symbol}
\begin{split}
[{\widehat I}, I,j]=0
&\ \Leftrightarrow\ 
\rho(\bar\sig_j)=\Id_{\dbU_1}
\ \Leftrightarrow\ 
\bar\sig_j \hbox{ fixes } \dbQ(\sqrt{q_{i_1}})\\
&\ \Leftrightarrow\ 
\bar q_{i_1}\in\bigl(\dbF_{q_j}^\times\bigr)^2
\ \ \Leftrightarrow\ \ 
\Bigl(\frac{q_{i_1}}{q_j}\Bigr)=1.
\end{split}
\end{equation}
This is a variant of \cite{Morishita02}*{Th.\ 3.1.3} and \cite{Morishita04}*{Example 1.3.1}.

\subsection{The case $n=2$ -- the R\'edei symbol.} 
Let ${\widehat I}=\{i_1,i_2,i_3=\infty\}\subseteq J$ have size $3$, with $i_1,i_2\in J_\val$, and let $I=\{i_1,i_2\}$.
We assume that:
\begin{itemize}
\item
$q_{i_1}\equiv q_{i_2}\equiv1\pmod4$;
\item
$(q_{i_1}/q_{i_2})=(q_{i_2}/q_{i_1})=1$.
\end{itemize}
Proposition \ref{isolated in 1 mod 4} yields globalizations $\rho_1,\rho_2$ for $\calG_{\{i_1,\infty\},\{i_1\}}$,  $\calG_{\{i_2,\infty\},\{i_2\}}$, respectively.
By Lemma \ref{kappa and power residue symbols},
\[
\kappa_{q_{i_1}}(\sig_{i_2})=\Bigl(\frac{q_{i_1}}{q_{i_2}}\Bigr)=1, \quad
\kappa_{q_{i_2}}(\sig_{i_1})=\Bigl(\frac{q_{i_2}}{q_{i_1}}\Bigr)=1.
\]
Since $\sig_{i_1}\in\bar Z_{i_1}$, and by the definition of $\bar Z_{i_1}$, we have $\kappa_{q_{i_1}}(\sig_{i_1})=\kappa_{t_{i_1}}(\sig_{i_1})=1$, and similarly, $\kappa_{q_{i_2}}(\sig_{i_2})=\kappa_{t_{i_2}}(\sig_{i_2})=1$.
It follows from Proposition \ref{existence of a globalization} that  $\rho_1$ and $\rho_2$ are trivial on $\bar\sig_{i_1},\bar\sig_{i_2}$.
From Corollary \ref{lifting for Vn}  we obtain a (necessarily unique) surjective globalization $\rho\colon G_{{\widehat I}}\to\dbU_2$ for $\calG_{{\widehat I},I}$ which lifts both $\rho_1$ and $\rho_2$.
Then $\Gal(E_\rho/\dbQ)\isom\dbU_2$.

By Theorem \ref{ramification}, the only prime numbers which ramify in $E_\rho$ are $q_{i_1},q_{i_2}$, with ramification groups $\dbZ/2$.
We now apply the following result of Amano:

\begin{thm}
\label{Amanos theorem}
(Amano \cite{Amano14a})  
For distinct prime numbers $q_{i_1}\equiv q_{i_2}\equiv1\pmod4$ such that $(q_{i_1}/q_{i_2})=(q_{i_2}/q_{i_1})=1$, there is a unique Galois extension $K_{\{q_{i_1},q_{i_2}\}}$ of $\dbQ$ with Galois group $\dbU_2$ in which $q_{i_1},q_{i_2}$ have ramification indices $2$ and all other prime numbers are unramified.
\end{thm}
Moreover, as shown by Amano, the field $K_{\{q_{i_1},q_{i_2}\}}$ is the \textsl{R\'edei extension}, constructed explicitly in \cite{Redei39}.
It follows that
\[
E_\rho=K_{\{q_{i_1},q_{i_2}\}}.
\]

Assume further that $j\in J_\val\setminus I$, $q_j\equiv1\pmod4$, and the image of $Z_j$ under $\rho$ is contained in $\dbV_2$.
In particular, $\Char\,\bar F_{v_j}\neq2$, so $Z_j=T_j\rtimes\bar Z_j=T_j\rtimes\langle\bar\sig_j\rangle$.
The image of $[{\widehat I},I,j]$ in $\dbZ/2$ is generated by $\rho(\bar\sig_j)_{13}$.
The \textsl{R\'edei triple symbol} $[q_{i_1},q_{i_2},q_j]_{\rm Redei}$ is defined as follows \cite{Amano14a}*{Def.\ 3.1}:
\[
[q_{i_1},q_{i_2},q_j]_{\rm Redei}=\begin{cases}
\ \ 1,& q_j\hbox{ is completely decomposed in } K_{\{q_{i_1},q_{i_2}\}},\\
-1,&\hbox{otherwise}.
\end{cases}
\]
Consequently,
\begin{equation}
\label{formula for Redei symbol}
\begin{split}
[{\widehat I},I,j]=0
&\ \Leftrightarrow\ 
\bar\sig_j \hbox{ fixes } E_\rho  
\ \Leftrightarrow\ 
\bar\sig_j \hbox{ fixes }K_{\{q_{i_1},q_{i_2}\}}\\
&\ \Leftrightarrow\ 
[q_{i_1},q_{i_2},q_j]_{\rm Redei}=1.
\end{split}
\end{equation}
This is a reformulation of Morishita's \cite{Morishita02}*{Th.\ 3.2.5}.
See also \cite{Vogel05} and \cite{Gartner11} for an additional analysis of this connection.

Numerical examples where the R\'edei symbol is nontrivial are given in \cite{Morishita02}*{Section 3} and \cite{Vogel05}*{Section 3}.
E.g., for $q_{i_1}=5$, $q_{i_2}=41$, $q_j=61$ all Legendre symbols above are trivial, and $[q_{i_1},q_{i_2},q_j]_{\rm Redei}=-1$

\subsection{Linking invariants for $n\geq3$.}
In a similar manner, one obtains linking invariants for higher values of $n$. 
For instance, take distinct $i_1,i_2,i_3\in J_\val$ and set $I=\{i_1,i_2,i_3\}$ and ${\widehat I}=\{i_1,i_2,i_3,i_4=\infty\}$.
Suppose that:
\begin{itemize}
\item
$q_{i_1}\equiv q_{i_2}\equiv q_{i_3}\equiv1\pmod4$;
\item
The Legendre symbols satisfy $\Bigl(\dfrac{q_{i_k}}{q_{i_l}}\Bigr)=1$ for all distinct $1\leq k,l\leq 3$;
\item
The R\'edei symbols satisfy $[q_{i_k},q_{i_l},q_{i_r}]_{\rm Redei}=1$ for all distinct $1\leq k,l,r\leq 3$.
\end{itemize}
Set ${\widehat I}_1=\{i_1,i_2,i_4=\infty\}$ and ${\widehat I}_2=\{i_2,i_3,i_4=\infty\}$.
As we have seen, there are surjective globalizations $\rho_1$, $\rho_2$ for $\calG_{{\widehat I}_1,\{i_1,i_2\}},\calG_{{\widehat I}_2,\{i_2,i_3\}}$, respectively.
Moreover, by the triviality of the R\'edei symbols, $\rho_1,\rho_2$ are trivial on the images of $\bar Z_{i_1},\bar Z_{i_2},\bar Z_{i_3}$.
By Corollary \ref{lifting for Vn}, $\rho_1,\rho_2$ lift further to a surjective globalization $\rho$ for $\calG_{{\widehat I},I}$.
Then $\Gal(E_\rho/\dbQ)\isom\dbU_3$.
By Theorem \ref{ramification}, the only prime numbers which ramify in $E_\rho$ are $q_{i_1},q_{i_2},q_{i_3}$, with ramification groups $\dbZ/2$.

Now let $j\in J_\val\setminus {\widehat I}$ with $q_j\neq2$.
Assuming that the image of $Z_j$ under $\rho$ is contained in $\dbV_3$, we have the linking invariant $[{\widehat I},I,j]\colon Z_j\to\dbZ/2$.
Its image is generated by $\rho(\bar\sig_j)_{14}$.
Thus $[{\widehat I},I,j]=0$ if and only if $\bar\sig_j$ fixes $E_\rho$.

Explicit constructions of Galois extensions with Galois group $\dbU_3$ are given in \cite{Amano14b} and \cite{AtaeiMinacTan17}.
We are unaware of such constructions for $n\geq4$.
Moreover, even for $n=3$, we do not know if a $\dbU_3$-Galois extension of $\dbQ$ with the above ramification properties is unique, similarly to Theorem \ref{Amanos theorem}.
Using the method described in this section, such constructions with uniqueness results will give rise to higher order \textsl{``unitriangular symbols''}, similar to the Legendre and R\'edei symbols for $n=1,2$, respectively.

\section{Cohomological Interpretation of Linking Invariants}
\label{section on cohomological interpretation of linking invariants}
In \cite{Morishita04}, Morishita relates his arithmetic Milnor invariants to Massey products in profinite cohomology.
In this spirit we now interpret the linking invariants $[I,j]$ cohomologically, and identify them as $n$-fold Massey product elements. 

\subsection{Profinite cohomology}
First, we fix some cohomological notation.
Given a $p$-power $q$, a pro-$p$ group $S$ and $r\geq0$, we write $H^r(S):=H^r(S,\dbZ/q)$ for the profinite cohomology group of $S$ of degree $r$ with respect to its trivial action of $S$ on $\dbZ/q$.
Recall that $H^1(S)$ consists of all pro-$p$ group homomorphisms $S\to\dbZ/q$.
We write $\inf,\res,\trg$ for the inflation, restriction, and transgression homomorphisms, respectively \cite{NeukirchSchmidtWingberg}*{Ch.\ I, Sect.\ 5}.

Consider an extension of pro-$p$ groups
\begin{equation}
\label{extension}
1\to N\to\hat T\xrightarrow{\hat\lam}G\to 1
\end{equation}
where $\hat T$ is a free pro-$p$ group and $N\leq \hat T^q[\hat T,\hat T]$.
Then the inflation map $\inf\colon H^1(G)\to H^1(\hat T)$  is an isomorphism and $H^2(\hat T)=0$ \cite{NeukirchSchmidtWingberg}*{Prop.\ 3.5.17}.
It follows from the 5-term exact sequence in profinite cohomology \cite{NeukirchSchmidtWingberg}*{Prop.\ 1.6.7} associated with (\ref{extension}) that the transgression map $\trg\colon H^1(N)^G\to H^2(G)$ is an isomorphism.

Also, there is a perfect bilinear map
\[
(\cdot,\cdot)'\colon N/N^q[N,G]\times H^1(N)^{G}\to\dbZ/q, \quad (\bar r,\psi)'=\psi(r),
\]
where $r\in N$ \cite{EfratMinac11}*{Cor.\ 2.2}.
It induces a bilinear map
\begin{equation}
\label{pairing}
(\cdot,\cdot)\colon N\times H^2(G)\to\dbZ/q, \quad (r,\alp)=(\trg\inv(\alp))(r),
\end{equation}
with left kernel $N^q[N,N]$ and a trivial right kernel.

For $m\geq2$, and with notation as in Example \ref{Sr}(2), we have a central extension
\[
0\to\dbV_m\to\dbU_m\to\overline{\dbU}_m\to1.
\]
Let $\omega_m\in H^2(\overline{\dbU}_m)$ be the classifying cohomology element for this extension via the Schreier correspondence  \cite{NeukirchSchmidtWingberg}*{Th.\ 1.2.4}. 

Let $\rho\colon G\to\dbU_m$ be a pro-$p$ group homomorphism, and set
\[
\hat\pi=\rho\circ\hat\lam\colon \hat T\to\dbU_m, \quad
\bar\rho=\pr_{\dbU_m,\overline\dbU_m}\circ\rho\colon G\to\overline{\dbU}_m.
\]
We write $\bar\rho^*(\omega_m)$ for the pullback of $\omega_m$ to $H^2(G)$ along $\bar\rho$.

\begin{prop}
\label{the Hoechsmann formula}
On $N$ one has  $(\cdot,\bar\rho^*(\omega_m))=\pr_{1,m+1}\circ\hat\pi|_N$.
\end{prop}
\begin{proof}
There is a commutative diagram with exact rows 
\[
\label{cd for extensions}
\xymatrix{
&&&\hat T\ar[ld]_{\hat\pi\times\hat\lam}\ar[ldd]^(0.25){\hat\pi}\ar[d]^{\hat\lam}&\\
1\ar[r]&\dbZ/q\ar[r]\ar[d]^{\wr}&\dbU_m\times_{\bar\dbU_m}G\ar[r]\ar[d]&G\ar[dl]^\rho\ar[r]\ar[d]^{\bar\rho}&1\\
1\ar[r]&\dbV_m\ar[r]&\dbU_m\ar[r]&\overline\dbU_m\ar[r]&1,
}
\]
where the vertical isomorphism is the inverse of $\pr_{1,m+1}$.
The pullback $\bar\rho^*(\omega_m)$ corresponds to the upper central extension \cite{EfratMinac11}*{Remark 6.1}.
One has $\hat\pi|_N\in H^1(N)^G$ and $\trg(\hat\pi|_N)=\bar\rho^*(\omega_m)$, by  \cite{Efrat23}*{Prop.\ 4.1 and Remark 4.2(4)} (which is based on results of Hoechsmann \cite{Hoechsmann68}).
The assertion now follows from the definition of the pairing $(\cdot,\cdot)$.
\end{proof}

\subsection{Linking invariants and relations in link groups}
For the cohomological description of the linking invariant $[{\widehat I},I,j]$, we will need the technical Lemma \ref{linking invariant and relations of link type groups} below, which gives a subtle connection between $[{\widehat I},I,j]$ and relations in linking structures of $\dbV_n$-link type, as in Example \ref{example of link type group}.

Consider the setup of Section \ref{section on the structures GI} with $R=\dbZ/q$.
Thus $I\subseteq{\widehat I}$, $j\not\in {\widehat I}$, and the images of $T_i$, $i\in I$, generate $G_{{\widehat I}}$ (assumption (A)).
We write 
\[
I^*=I\cup\{j\}=\{i_1\nek i_n,j\}, \quad {\widehat I}^*={\widehat I}\cup\{j\}.
\]
In this way we obtain a linking structure $\calG_{{\widehat I}^*,I^*}$ with target group $\dbU_{n+1}$, and we define 
\[
G_{{\widehat I}^*}, \ \hat T^*=\star_{i\in I^*}T_i,\  \res^*_j\colon Z_j\to G_{{\widehat I}^*},\  \hat\lam^*_{T^*}\colon\hat T^*\to G_{{\widehat I}^*},\  \hat\pi^*\colon \hat T^*\to\dbU_{n+1}
\]
accordingly.
Let $\hat\lam_{T^*}\colon\hat T^*\to G_{{\widehat I}}$ be the composition of $\hat\lam^*_{T^*}$ with the epimorphism $G_{{\widehat I}^*}\to G_{{\widehat I}}$. 

Suppose that $\rho^*$ is a globalization for $\calG_{{\widehat I}^*,I^*}$.
By Remark \ref{uniqueness}, it induces the above globalization $\rho$ for $\calG_{{\widehat I},I}$, and there is a commutative diagram
\[
\xymatrix{
\hat T^*\ar@/_2pc/[dd]_{\hat\lam_{T^*}}\ar[d]^{\hat\lam^*_{T^*}}\ar[rd]^{\hat\pi^*}&\\
G_{{\widehat I}^*}\ar[d]\ar[r]^{\rho^*}&\dbU_{n+1}\ar[d]^{\pr}\\
G_{{\widehat I}}\ar[r]^{\rho}&\dbU_n.
}
\]
Here the map $\pr\colon \dbU_{n+1}\to\dbU_n$ is the projection on the upper-left $(n+1)\times(n+1)$-submatrix.

We recall that the linking invariant $[{\widehat I},I,j]$ is defined only under the assumption that $\rho$ maps the image $\res_j(Z_j)$ of $Z_j$ in $G_{{\widehat I}}$ into $\dbV_n$.

\begin{lem}
\label{linking invariant and relations of link type groups}
Assume that $(\rho\circ\res_j)(Z_j)\subseteq \dbV_n$.
Let $\sig\in Z_j$, $\tau',\hat\sig\in\hat T^*$ and $\tau\in T_j$, and suppose that $\hat\lam^*_{T^*}(\hat\sig)=\res^*_j(\sig)$.
For $r=(\tau')^q[\tau,\hat\sig]$ one has
\[
(\pr_{1,n+2}\circ\hat\pi^*)(r)=-\bar\pi_j(\tau)\cdot [{\widehat I},I,j](\sig).
\]
\end{lem}
\begin{proof}
First we note that, by (\ref{images of pi j}), 
\[
(\rho^*\circ\hat\lam^*_{T^*})(\tau')=\hat\pi^*(\tau')\in\Id_{\dbU_{n+1}}+\sum_{l=1}^{n+1}(\dbZ/q)E_{l,l+1}.
\]
The latter subgroup of $\dbU_{n+1}$ has exponent $q$, so $(\rho^*\circ\hat\lam^*_{T^*})((\tau')^q)=0$.

Next, by our assumption on $\rho$,
\[
(\pr\circ\rho^*\circ\hat\lam^*_{T^*})(\hat\sig)=(\pr\circ\rho^*\circ\res^*_j)(\sig)=(\rho\circ\res_j)(\sig)\in\dbV_n.
\]
Hence the matrix $(\rho^*\circ\lam^*_{T^*})(\hat\sig)$ has zero entries outside the main diagonal, the $(n+2)$-th column, and entry $(1,n+1)$.
Moreover, its $(1,n+1)$-entry is $[{\widehat I},I,j](\sig)$.
A direct computation of matrix commutators gives
\[
\begin{split}
\hat\pi^*(r)&=
(\rho^*\circ\hat\lam^*_{T^*})((\tau')^q[\tau,\hat\sig])=[(\rho^*\circ\res^*_j)(\tau),(\rho^*\circ\hat\lam^*_{T^*})(\hat\sig)] \\
=&[\hat\pi^*(\tau),(\rho^*\circ\hat\lam^*_{T^*})(\hat\sig)] \\
=&
\begin{bmatrix}
\begin{bmatrix}
1&0&0&\cdots&0&0\\
&1&0&\cdots&0&0\\
&&&\ddots&0&0\\
&&&&1&\bar\pi_j(\tau)\\
&&&&&1
\end{bmatrix}
,&
\begin{bmatrix}
1&0&0&\cdots&[{\widehat I},I,j](\sig)&*\\
&1&0&\cdots&0&*\\
&&&\ddots&0&*\\
&&&&1&*\\
&&&&&1
\end{bmatrix}
\end{bmatrix}
\\
=&\Id-\bar\pi_j(\tau)\cdot[{\widehat I},I,j](\sig)\cdot E_{1,n+2},
\end{split}
\]
where $E_{1,n+2}$ is as in Subsection \ref{subsection on preliminaries}.
The assertion now follows by examining the $(1,n+2)$-th entries.
\end{proof}

\subsection{Linking invariants and Massey products}
The linking structures of Example \ref{example of link type group}, on which our constructions of linking invariants were based, had presentations with defining relations as in Lemma \ref{linking invariant and relations of link type groups}.
We now apply Proposition \ref{the Hoechsmann formula} to relate such relations to the cohomology elements $\bar\rho^*(\omega_{n+1})$ and the linking invariant $[{\widehat I},I,j]$:

\begin{cor}
\label{linking invariant and pullbacks}
Consider the setup of Lemma \ref{linking invariant and relations of link type groups}.
Assume that $T_i\isom\dbZ_p$ for every $i\in I^*$ and that $N=\Ker(\hat\lam^*_{T^*})$ is generated as a closed normal subgroup of $\hat T^*$ by elements $r=(\tau')^q[\tau,\hat\sig]$ as in Lemma \ref{linking invariant and relations of link type groups}.
Then for every such $r$ one has
\[
(r,\bar\rho^*(\omega_{n+1}))=-\bar\pi_j(\tau)\cdot[{\widehat I},I,j](\sig),
\]
where $\bar\rho=\pr_{\dbU_{n+1},\overline{\dbU}_{n+1}}\circ\rho^*$ and the left-hand side is the pairing (\ref{pairing}).

\end{cor}
\begin{proof}
The assumption on the $T_i$ implies that their free pro-$p$ product $\hat T^*$ is a free pro-$p$ group.
The assumption on $N$ implies that $N\leq (\hat T^*)^q[\hat T^*,\hat T^*]$.
We may therefore apply Proposition \ref{the Hoechsmann formula} with $m=n+1$ and the central extension
\[
1\to N\to\hat T^*\xrightarrow{\hat\lam^*_{T^*}} G_{{\widehat I}^*}\to1,
\]
to obtain that 
\[
(r,\bar\rho^*(\omega_{n+1}))=(\pr_{1,n+2}\circ\hat\pi^*)(r).
\]
The assertion follows from Lemma \ref{linking invariant and relations of link type groups}.
\end{proof}

Massey products are defined in the general homological context of differential graded algebras.
In the special context of group cohomology and dimension $2$, they have the following convenient alternative definition, which is due to Dwyer \cite{Dwyer75} for discrete groups; see \cite{Efrat14}*{Sect.\ 8} for the profinite analog:

\begin{defin}
\rm
Given a pro-$p$ group $S$ and $\varphi_1\nek\varphi_m\in H^1(S)$, $m\geq2$, the \textsl{$m$-fold Massey product} $\langle\varphi_1\nek\varphi_m\rangle$ consists of all pullbacks
$\bar\rho^*(\omega_m)$ in $H^2(S)$, where $\bar\rho$ ranges over all pro-$p$ group homomorphisms $S\to\overline{\dbU}_m$ such that $\varphi_l=\pr_{l,l+1}\circ\bar\rho$, $l=1,2\nek m$.
\end{defin}

Thus in Corollary \ref{linking invariant and pullbacks}, the pullback $\bar\rho^*(\omega_{n+1})$ is an element of the Massey product $\langle\varphi_1\nek\varphi_m\rangle$, where $\varphi_l=\pr_{l,l+1}\circ\rho^*$, $l=1,2\nek n+1$.
On $T_k$, $k\in I^*$, we have, by (\ref{images of pi j}),
\[
\varphi_l\circ\res_k=\pr_{l,l+1}\circ\rho^*\circ\res_k
=\pr_{l,l+1}\circ\pi_k
=\begin{cases}\bar\pi_l,&\hbox{if }k=i_l,\\
0,&\hbox{otherwise.}
\end{cases}
\]
Since $\res_k(T_k)$, $k\in I^*$, generate $G_{I^*}$, this determines $\varphi_l$.

\end{document}